\documentclass{amsart}
\usepackage[english]{babel}
\usepackage[latin1]{inputenc}
\usepackage[dvips,final]{graphics}
\usepackage{amsmath,amsfonts,amssymb,amsthm,amscd,array,stmaryrd,mathrsfs, mathdots, epigraph}
\usepackage{pstricks}
 \usepackage[all]{xy}
 \usepackage{url}
\usepackage{multirow, blkarray}
\usepackage{booktabs}
\usepackage{textcomp}
 \usepackage[final]{epsfig}
 \usepackage{color}
\theoremstyle{plain}
\newtheorem{thm}{Theorem}
\newtheorem{lem}{Lemma}[section]
\newtheorem{cor}[lem]{Corollary}

\newtheorem{defn}[lem]{Definition}

\theoremstyle{definition}
\newtheorem{rem}[lem]{Remark}
\newtheorem{ex}[lem]{Example}

\newcommand{\Z}{\mathbb{Z}}
\newcommand{\C}{\mathbb{C}}
\newcommand{\N}{\mathbb{N}}

\newcommand{\Qc}{\mathcal{Q}} 
\newcommand{\Rc}{\mathcal{R}}

\newcommand{\id}{\textup{Id}}

\newcommand{\SL}{\mathrm{SL}}

\newcommand{\OSp}{\mathrm{OSp}}
\newcommand{\SSL}{\mathrm{SL}}

\def\a{\alpha}
\def\b{\beta}
\def\d{\delta}

\def\e{\varepsilon}

\def\g{\gamma}

\def\om{\omega}

\begin{document}

\title{A step towards cluster superalgebras}

\author{Valentin Ovsienko}

\address{
Valentin Ovsienko,
CNRS,
Laboratoire de Math\'ematiques 
U.F.R. Sciences Exactes et Naturelles 
Moulin de la Housse - BP 1039 
51687 REIMS cedex 2,
France}

\email{valentin.ovsienko@univ-reims.fr}



\begin{abstract}
We introduce a class of commutative superalgebras 
generalizing cluster algebras.
A cluster superalgebra is defined by a hypergraph called an ``extended quiver'', 
and  transformations called mutations.
We prove the super analog of the ``Laurent phenomenon'', 
i.e., that all elements of a given cluster superalgebra are Laurent polynomials
in the initial variables,
and find an invariant presymplectic form.
Examples of cluster superalgebras are provided by superanalogs of
Coxeter's frieze patterns.
We apply the Laurent phenomenon to
construct a new integer sequence extending the Somos-$4$ sequence.
\end{abstract}

\maketitle

\tableofcontents

\section*{Introduction}

Commutative superalgebras appear naturally in classical geometry.
One classical example is the classical Grassmann algebra of differential forms
on the manifold.
Commutative superalgebras are an essential part of any cohomology theory. 
Lie superalgebras also appear in physics where they are extensively used
to describe symmetries of elementary particles. 
Theory of supermanifolds was developped in 70's 
by Berezin and Leites, see~\cite{Ber,DM,Man1}
 as well as \cite{Lei,Man,LeiR}.
This theory offers 
geometry and analysis based on spaces ringed by superalgebras.
Commutative superalgebras can be regarded as a 
special class of non-commutative or ``quantum'' algebras,
their main feature is the presence of odd, or ``Grassmann''
elements that square to zero.

\medskip

\paragraph{\bf Classical cluster algebras}
Cluster algebras, discovered by Fomin and Zelevinsky~\cite{FZ1}, 
are a special class of commutative associative algebras.
A cluster algebra is a subalgebra of the algebra generated by
Laurent polynomials with positive integer coefficients in
$\Z[x^{\pm1}_1,\ldots,x^{\pm1}_n]$.
A cluster algebra is usually defined with the help of a quiver (an oriented graph)
with no loops and no $2$-cycles;
the generators of the algebra are defined with the help of
{\it exchange relations} and {\it mutations} of the initial quiver.
More precisely, the vertices of the initial quiver are labeled by
$\{x_1,\ldots,x_n\}$;
the mutation at a chosen variable~$x_k$
changes the quiver and replaces $x_k$ with the new variable
$$
x_k'=\frac{M_1+M_2}{x_k},
$$
where $M_1$ and $M_2$ are the monomials
obtained as products of variables connected to~$x_k$
by ingoing and outgoing arrows, respectively.
The above formula is called an
exchange relation; a cluster algebra is generated by all possible
mutations and exchange relations.

Mutations of the quiver at a chosen vertex 
are transformations that can be illustrated by the following diagram:
$$
 \xymatrix{
\bullet\ar@{->}[r]&\star&
\bullet\ar@{<-}[l]
}
\qquad
\Longrightarrow
\qquad
 \xymatrix{
\bullet\ar@{<-}[r]\ar@/^-0.7pc/[rr]&\star&
\bullet\ar@{->}[l]\\
}
$$
The most important point is the creation of a new arrow\footnote{
Exactly the same mutation rules are intensively used in physics
under the name of Seiberg Duality, cf.~\cite{BD},
for a recent treatment, see~\cite{Phy} and references therein.}.

Cluster algebras naturally appear in algebra, geometry and combinatorics.
The algebras of regular functions of many algebraic varieties, related to Lie theory, have cluster structure.
The main examples are simple Lie groups~\cite{BFZ,GSV2},
Grassmannians~\cite{Sco}, flag varieties, etc.
Cluster algebras are closely related to
symplectic and Poisson geometry~\cite{GSV1,GSV},
representation theory, Teichm\"uller spaces~\cite{FG} and various moduli spaces.
For more information and references, see surveys~\cite{Mar,Wil}.

Cluster algebras are also closely related to
integrable systems.
The Hirota (or octahedral) recurrence, see~\cite{Zab},
which is a famous and ``universal'' discrete integrable system,
is an example of cluster exchange relations.
Many spaces on which discrete integrable systems act,
such as moduli spaces of configurations of points in projective spaces,
have cluster structures, see~\cite{MGOT,SVRS}.
Among a wealth of notions related to these moduli spaces, we mention
frieze patterns, see~\cite{Cox,CaCh}.
Moduli spaces of configurations of points are also related to 
the spaces of linear difference equations with special periodicity conditions;
for a detailed account, see~\cite{SVRS}.
For recent developments in this area, 
see~\cite{FM,GSTV1,Gli,FH,MG,GK,OST,Kr}.

\medskip

\paragraph{\bf Towards cluster superalgebras}
A challenging problem is to develop a notion of
cluster superalgebra.
This paper is a first step in this direction.
The general idea is to include odd, or Grassmann, indeterminates
$\{\xi_1,\ldots,\xi_m\}$, that anticommute with each other,
and in particular, square to zero.
The cluster superalgebra is a certain subalgebra of
$$
\C[x^{\pm1}_1,\ldots,x^{\pm1}_n,\xi_1,\ldots,\xi_m].
$$
A cluster superalgebra is defined by an ``extended quiver'' which is a hypergraph extending
 a classical quiver.
The main ingredients are modified exchange relations and quiver mutations.
The vertices of the classical quiver are labeled by the even variables,
the new vertices are labeled by the odd variables.
Essentially, the mutations of an extended quiver
(observe that the two upper vertices are colored while the three lower ones are black) are defined
by the following rule:
$$ 
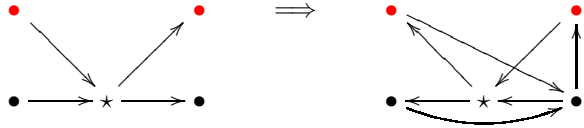
\xymatrix{
{\color{red}\bullet}\ar@<1pt>@{->}[rd]&&
{\color{red}\bullet}\ar@<-1pt>@{<-}[ld]\\
\bullet&\star\ar@{<-}[l]&\bullet\ar@{<-}[l]
}
\qquad
\Longrightarrow
\qquad
 \xymatrix{
{\color{red}\bullet}\ar@<1pt>@{<-}[rd]\ar@<1pt>@{->}[rrd]&&
{\color{red}\bullet}\ar@<-1pt>@{->}[ld]\ar@{<-}[d]\\
\bullet\ar@/^-0.7pc/[rr]&\star\ar@{->}[l]&\bullet\ar@{->}[l]
}
$$
\smallskip

\noindent
Note that this rule has similarity with mutation of colored graphs,
see~\cite{BT,MP}.

\medskip

\paragraph{\bf Comments}
It should be stressed that the algebras constructed in this paper 
are not particular cases of known versions of cluster algebras.
In particular a cluster super algebra is not a quantum cluster algebra~\cite{BZ},
or other known non-commutative analogs of cluster algebras.

We believe that the definition of cluster superalgebra can be more general.
The main limitations of our construction are:

-absence of arrows between odd vertices;

-lack of exchange relations for the odd variables;

-lack of geometric examples.

\medskip

It would be interesting to look for more general definitions,
as well as for more examples of cluster superalgebras and related
supermanifolds.
In particular, it would be interesting to look for structures of cluster superalgebra
on the space of regular functions of classical supergroups, similar to 
cluster algebra structures of \cite{BFZ}.

It would also be interesting to look for examples related to super Lie theory.
However, a naive attempt to use Berezinians and condensation formulas
encounters some difficulties.
The series of supergroups $\SL(n|m)$ lead to the classical cluster algebras, 
while the case of $\OSp(n|m)$ is much more difficult.

Our approach is based on relations of cluster algebras to moduli spaces, 
linear difference equations,
and frieze patterns.
Supersymmetric linear difference equations
 with some special monodromy conditions and the corresponding {\it superfriezes},
analogous to Coxeter's frieze patterns, were
recently introduced in~\cite{SFriZ}.
Our definition of cluster superalgebras suggested in the present paper
is a result of analysis of superfriezes.
We consider a class of examples of cluster superalgebras ``of type $A$'' that
correspond to superfriezes.

\medskip

\paragraph{\bf Structure of the paper}

The paper consists of seven sections.

In introductory Section~\ref{ClasSec}, we briefly review classical cluster algebras,
and show how to recover the rules of mutation of a quiver from
the canonical presymplectic form and exchange relations.

In Section~\ref{DeFCluSS}, we introduce the notion of extended quiver
and the rules of mutation.

In Section~\ref{Echangist}, we define the exchange relations
generalizing the classical exchange relations to the presence of odd variables.
We also introduce the presymplectic form and formulate the main results of the paper:
the Laurent phenomenon and 
invariance of the presymplectic form.
We finally define cluster superalgebras.

In Section~\ref{ElExSec},
we consider a list of examples.

In Section~\ref{SFS}, we present an infinite series of examples
of cluster superalgebras, related to the notion of
``superfriezes'' and difference equations studied in~\cite{SFriZ}.
Observe once more that superfriezes is the origin of the general notion of cluster superalgebra
developed in the present paper.

In Section~\ref{PrSect}, we prove 
the results formulated in Section~\ref{Echangist}.

In Sections~\ref{ApplSect} and ~\ref{FiboSect}, we present number theoretic applications.

\section{Classical cluster algebras and their presymplectic form}\label{ClasSec}
This section is a very elementary introduction to cluster algebras;
our main goal is to explain the role of the canonical presymplectic $2$-form.

A {\it cluster algebra of rank} $n$ is a commutative associative algebra
associated to a {\it quiver}, i.e., a  finite oriented graph, $\Qc$, with no loops and no $2$-cycles.
We use the common notation: $\Qc_0$ is the set of vertices,
and $\Qc_1$ is the set of arrows of $\Qc$.
The vertices of $\Qc$ are labeled by commuting variables
$\{x_1,\ldots,x_n\}$, where $|\Qc_0|=n$.

\subsection{The definition}\label{ClDeS}
The cluster algebra defined by $\Qc$ is generated by $\{x_1,\ldots,x_n\}$
and the additional rational functions $x'_k,\ldots$ obtained recursively by
a sequences of transformations called {\it mutations}:
$$
\mu_k:
\left(
\{x_1,\ldots,x_n\},\Qc
\right)\mapsto
\left(
\{x_1,\ldots,x_n\}\setminus\{x_k\}\cup\{x_k'\},\Qc'
\right).
$$
For every $k\in\{1,\ldots,n\}$, the mutation at vertex $x_k$
replaces the variable $x_k$ by the new 
function~$x_k'$ defined by the formula:
$$
x_kx'_k=\prod\limits_{\substack{x_k\to x_j}}\;x_j
 \quad+\quad 
\prod\limits_{\substack{x_i\to x_k}}\; x_i,
$$
where the products are taken over the set of arrows
$(x_i\rightarrow x_k)\in\Qc_1$ and $(x_k\to x_j)\in\Qc_1$, respectively
(with fixed $k$).
The above formula is
called an \textit{exchange relation}.

The mutation of the quiver $\mu_k:\Qc\to\Qc'$
at the vertex $x_k$ is defined
according to the following three rules:
\begin{enumerate}
\item for every path $(x_i\rightarrow x_k \rightarrow x_j)$ in $\Qc$, add an arrow $(x_i\rightarrow x_j)$;
\item reverse all the arrows leaving or arriving at $x_k$;
\item remove all 2-cycles created by rule (1).
\end{enumerate}

Every mutation is an involution, i.e., $\mu_k^2=\id$, but two different mutations
do not necessarily commute.
Applying all possible mutations,
one obtains (possibly infinitely many) different quivers, and rational functions $x'_k,\ldots$
in the initial coordinates $\{x_1,\ldots,x_n\}$.
The cluster algebra associated with the initial quiver $\Qc$,
is the associative commutative algebra generated by all the functions obtained
in the mutation process.

Some of the initial cluster coordinates are called {\it frozen},
and remain unchanged.
One simply does not consider mutations at frozen coordinates.
These coordinates are additional parameters
of the cluster algebra.

\medskip

Among the main results of Fomin and Zelevinsky are:

\begin{itemize}
\item
\textit{Laurent phenomenon}:  
all elements of a given cluster algebra are Laurent polynomials in
 $\{x_1,\ldots,x_n\}$, see~\cite{FZ1},\cite{FZ2}.
\item
\textit{A finite type classification}:
a cluster algebra is of finite type, i.e., finitely generated,
if and only if the quiver $\Qc$ is equivalent under a sequence of
mutations to one of the Dynkin quivers of type~$A,D$, or~$E$, see~\cite{FZ3}.
\end{itemize}

\noindent
Furthermore, the {\it positivity property} conjectured in~\cite{FZ1}
and recently proved in \cite{LS},
states that the coefficients of the Laurent monomials in 
all the generators of a cluster algebra are non-negative integers.

\subsection{The canonical presymplectic form}\label{GSVFSec}
Every cluster algebra (and thus every algebraic manifold 
whose algebra of regular functions has a cluster algebra structure) 
has a canonical presymplectic differential $2$-form
introduced by Gekhtman, Shapiro and Vainshtein~\cite{GSV1}.
It is given by
$$
\om=
\sum\limits_{\substack{x_i\to{}x_j}}
\frac{dx_i\wedge{}dx_j}{x_ix_j},
$$
where the summation is performed over the the arrows
$(x_i\rightarrow x_k)\in\Qc_1$.
This form is invariant with respect to the cluster mutations
and therefore is well-defined globally, on every cluster variety.
Note that the form $\om$ is obviously closed since it is with constant coefficients
in coordinates $\log{}x_i$.
However, it is not necessarily non-degenerate, for instance, if~$n$ is odd.

Quite remarkably, among the following three data:
$$
\left(
\hbox{exchange relations},
\qquad
\hbox{quiver mutations},
\qquad
\hbox{canonical $2$-form $\om$}
\right)
$$
every two contain the full information, and allow one to recover the third ingredient.

Let us show how to recover the quiver mutations
from the exchange relations and the invariant $2$-form $\om$.
In order to make the exposition elementary,
we first consider a simple example.

\begin{ex}
Consider the cluster algebra associated to the Dynkin quiver $A_3$:
$$
 \xymatrix{
x\ar@{->}[r]&y&
z\ar@{<-}[l]
}
$$
The corresponding $2$-form is
$$
\om=\frac{dx\wedge{}dy}{xy}+\frac{dy\wedge{}dz}{yz},
$$
and the exchanging relation for $y$ is
$y'=(x+z)/y$.
Expressing then $y=(x+z)/y'$, and substituting into $\omega$, one obtains
from the Leibniz rule:
$$
\begin{array}{rcl}
\om&=&
\displaystyle
\frac{dx\wedge{}dz}{x(x+z)}
-\frac{dx\wedge{}dy'}{xy'}+\frac{dx\wedge{}dz}{(x+z)z}
-\frac{dy'\wedge{}dz}{y'z}\\[10pt]
&=&
\displaystyle
\frac{dx\wedge{}dz}{xz}+\frac{dy'\wedge{}dx}{xy'}+\frac{dz\wedge{}dy'}{y'z}.
\end{array}
$$
Invariance of $\om$ means that, after the above coordinate
transformation, $\om$ is exactly the form of the quiver 
$$
 \xymatrix{
x\ar@{<-}[r]\ar@/^-0.7pc/[rr]&y'&
z\ar@{->}[l]\\
}
$$
obtained as a result of the mutation at
the vertex $y$.
\end{ex}

The general situation is similar.
Rewrite the form~$\om$ 
using the summation over vertices of~$\Qc$. 
One immediately obtains from the Leibniz rule:
$$
\om=\frac12
\sum\limits_{\substack{x_k\in\Qc_0}}
\frac{d\frac{x_{I_k}}{x_{J_k}}\wedge{}dx_k}{\frac{x_{I_k}}{x_{J_k}}\,x_k},
$$
where 
$$
x_{I_k}=x_{i_1}\cdots{}x_{i_s},
\qquad
x_{J_k}=x_{j_1}\cdots{}x_{j_t}
$$ 
are the monomials obtained as the products of coordinates connected to $x_k$
by ingoing arrows: $(x_i\to{}x_k)$, and by outgoing arrows:
$(x_k\to{}x_j)$, respectively.

Invariance of the latter formula for $\om$ under mutations at $x_k$ is quite obvious.
Indeed, temporarily using the notation $\widetilde{x_k}':=\frac{x_k'}{x_{J_k}}$,
rewrite the exchange relations in the form:
$$
x_k=
\frac{1}{\widetilde{x_k}'}\left(1+\frac{x_{I_k}}{x_{J_k}}\right).
$$
Let also
$$
\om_{x_k}=\frac12\frac{d\frac{x_{I_k}}{x_{J_k}}\wedge{}dx_k}{\frac{x_{I_k}}{x_{J_k}}\,x_k}
$$ 
be the term of $\om$ containing $x_k$.
Substituting the above equation for $x_k$,
and by the Leibniz rule one immediately obtains:
$$
\begin{array}{rcl}
\om_{x_k}
&=&
\displaystyle
\frac12\frac{d\widetilde{x_k}'\wedge{}d\frac{x_{I_k}}{x_{J_k}}}{\widetilde{x_k}'\frac{x_{I_k}}{x_{J_k}}}\\[14pt]
&=&
\displaystyle
\frac12\left(
\frac{dx_k'\wedge{}d\frac{x_{I_k}}{x_{J_k}}}{x_k'\frac{x_{I_k}}{x_{J_k}}}-
\frac{dx_{J_k}\wedge{}dx_{I_k}}{x_{I_k}x_{J_k}}\right).
\end{array}
$$
The first term corresponds to the reversed arrows at $x_k'$,
while the second term corresponds to the extra arrows $(x_j\to{}x_i)$.
Invariance of $\om$ therefore implies the rules (1) and (2) of quiver mutations.

The Gekhtman-Shapiro-Vainshtein form $\om$, and related family of
Poisson structures, play an important role in applications.
It equips cluster algebraic varieties with
additional geometric structures.
A similar presymplectic form will be crucial for us to understand
mutations of extended quivers.

\section{Extended quivers and their mutations}\label{DeFCluSS}

We introduce extended quivers, and describe their mutation rules.
It turns out that an extended quiver is not a graph but
an oriented {\it hypergraph}.
More precisely, given a quiver $\Qc$,
we add new, colored (or odd), vertices, and complete
the set of edges of $\Qc$ by adding
some $2$-paths joining three vertices.

The reason for this notion is the general idea of superalgebra and supergeometry,
that supersymmetric version of every object should be understood as its ``square root''.
The notion of extended quiver is an attempt to apply this idea in combinatorics:
a square root of an edge in a graph is understood as a $2$-path joining two
odd vertices through an even vertex.

\subsection{Introducing extended quivers}\label{IntroQuiS}

\begin{defn}
\label{AllGraphDef}
Given a quiver $\Qc$  with no loops and no $2$-cycles,
an {\it extended quiver} $\widetilde{\Qc}$ with {\it underlying} quiver $\Qc$, 
is an oriented hypergraph defined as follows.
\begin{enumerate}
\item[{\bf (A)}]
The vertices of $\Qc$ are labeled by $\{x_1,\ldots,x_n\}$,
$\widetilde{\Qc}$ has $m$ extra ``colored'' vertices labeled by the odd variables 
$\{\xi_1,\ldots,\xi_m\}$, so that
$$
\widetilde{\Qc}_0={\Qc}_0\cup\{\xi_1,\ldots,\xi_m\}.
$$

\item[{\bf (B)}]
Some of the new vertices $\{\xi_1,\ldots,\xi_m\}$ are related 
by $2$-paths through
the vertices $\{x_1,\ldots,x_n\}$ of the underlying quiver $\Qc$.
More precisely,
for every $1\leq{}k\leq{}n$, two non-intersecting subsets 
$$
I_k,J_k\subset\{1,\ldots,m\},
\qquad
I_k\cap{}J_k=\emptyset
$$
are fixed; the set of arrows $\Qc_1$ is completed by the set of $2$-paths:
$$
\widetilde{\Qc}_1=\Qc_1\cup_{k}\left\{(\xi_i\to{}x_k\to\xi_j),\; i\in{}I_k,j\in{}J_k\right\},
$$
so that the following condition is satisfied.

\item[{\bf (C)}] 
Every vertex $\xi_i$, for $i\in{}I_k$, is connected to every vertex
$\xi_j$, for $j\in{}J_k$, by exactly the same number of $2$-paths through $x_k$. 
\end{enumerate}
\end{defn}

Although $\widetilde{\Qc}$ is a hypergraph, and therefore can hardly be represented
graphically, the above definition is illustrated by the following diagram.
To fix the notation, let 
$I_k=\{i_1,\ldots,i_r\},$
and
$J_k=\{j_1,\ldots,j_s\}.$
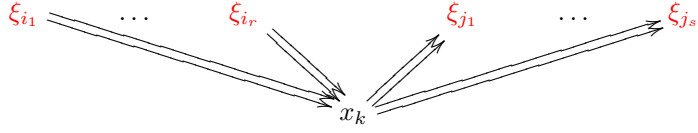
$$
 \xymatrix{
{\color{red}\xi_{i_1}}\ar@{->}[rrrd]\ar@<3pt>@{->}[rrrd]&\dots&
{\color{red}\xi_{i_r}}\ar@{->}[rd]\ar@<3pt>@{->}[rd]&&
{\color{red}\xi_{j_1}}\ar@{<-}[ld]\ar@<3pt>@{<-}[ld]&\dots&
{\color{red}\xi_{j_s}}\ar@{<-}[llld]\ar@<3pt>@{<-}[llld]\\
&&& x_k&&
}
$$
Here, {\it every} vertex $\xi_i\in\{\xi_{i_1},\ldots,\xi_{i_r}\}$ is connected to
{\it every} vertex $\xi_j\in\{\xi_{j_1},\ldots,\xi_{j_s}\}$ by a certain number of $2$-paths
$(\xi_i\to{}x_k\to\xi_j)$ through $x_k$.
This number is independent of the choice of $i\in I_k$ and $j\in J_k$.

\begin{rem}
(a)
Condition (C) is not just a technical restriction.
Our main results: the Laurent phenomenon
and invariant presymplectic form, hold if and only if this condition is satisfied.

(b)
Note that we do not consider arrows between the odd vertices of $\widetilde{\Qc}$,
and this is certainly an interesting question whether one can add such arrows and create a more rich
combinatorics of extended quivers.
\end{rem}

\subsection{Quiver mutations}\label{GraphMut}

Let us define the mutation rules of an extended quiver.
These mutations are performed at even vertices.
The rules of mutations are modifications of the classical rules (1), (2) (3),
see Section~\ref{ClDeS}.

\begin{defn}
\label{MutDef}
Given an extended quiver~$\widetilde{\Qc}$ and an even vertex $x_k\in\Qc_0$,
the mutation~$\mu_k$ is defined by the
following rules:
\begin{enumerate}
\item[(0)] 
The underlying quiver $\Qc\subset\widetilde{\Qc}$ mutates
according to the same rules (1), (2), (3) as in the classical case.
The $k$th label changes to $x'_k$ (whose new value will be defined later on).

\item[(1*)] 
Given a $2$-path $(\xi_i\to x_k\to\xi_j)\in\widetilde{\Qc}_1$,
for all $x_\ell\in\Qc_0$
connected to~$x_k$ by an outgoing arrow $(x_k\to{}x_\ell)$, 
add the $2$-paths $(\xi_i\to x_\ell\to\xi_j)$.

\item[(2*)]  Reverse all the $2$-paths through $x_k$, i.e.,
change $(\xi_i\to x_k\to\xi_j)$ to $(\xi_i\leftarrow x_k\leftarrow\xi_j)$.

\item[(3*)]  Remove two-by-two the $2$-paths through $x_k$ (if any)
with opposite orientations, created by rule (1*), i.e.,
$2$-paths $(\xi_i\to{}x_\ell\to\xi_j)$ and 
$(\xi_i\leftarrow{}x_\ell\leftarrow\xi_j)$ cancel each other.
\end{enumerate}
\end{defn}

The above rules can be illustrated by the diagram:
$$ 
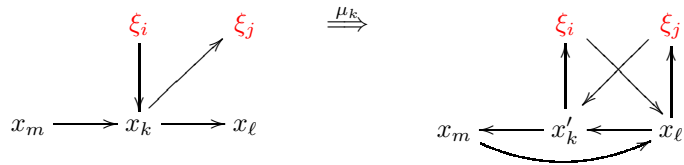
\xymatrix{
&{\color{red}\xi_i}\ar@{->}[d]&
{\color{red}\xi_j}\ar@<-2pt>@{<-}[ld]\\
x_m&x_k\ar@{<-}[l]&x_\ell\ar@{<-}[l]
}
\qquad
\stackrel{\mu_k}{\Longrightarrow}
\qquad
 \xymatrix{
&{\color{red}\xi_i}\ar@<2pt>@{->}[rd]\ar@{<-}[d]&
{\color{red}\xi_j}\ar@<-2pt>@{->}[ld]\ar@{<-}[d]\\
x_m\ar@/^-1pc/[rr]&x_k'\ar@{->}[l]&x_\ell\ar@{->}[l]
}
$$

\begin{rem}
To compare the new rule (1*) to the classical mutation rule (1), note that the 
new rule changes the set $\widetilde{\Qc}_1$ of $2$-paths, and not just the set
of arrows, as the old rule (1).
\end{rem}

\subsection{Allowed and forbidden mutations}

A mutation of the extended quiver $\widetilde{\Qc}$
at a vertex $x_k\in\Qc_0$ is {\it allowed} if, in the resulting quiver~$\mu_k(\widetilde{\Qc})$
satisfies Condition (C) from Section~\ref{IntroQuiS}.
The reason why some mutations can violate this condition,
and are therefore forbidden, is due to the fact that
the sets $I_\ell$ and $J_\ell$ attached to a vertex $x_\ell$ connected to
$x_k$ by an arrow $(x_k\to{}x_\ell)$ are transformed:
$$
\mu_k: 
\left\{
\begin{array}{l} 
I_\ell\Rightarrow{}I_\ell\cup{}I_k,\\[2pt]
J_\ell\Rightarrow{}J_\ell\cup{}J_k,
\end{array}
\right.
$$
as a result of the rule (1*).

\begin{ex}
a)
Mutation of the following quiver at $x_1$ is not allowed.
$$
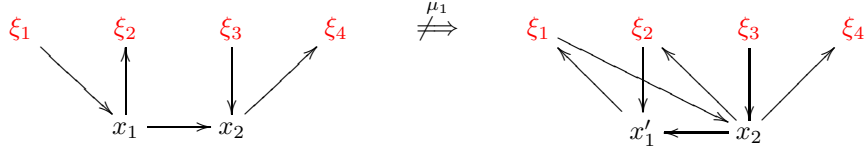
\xymatrix{
{\color{red}\xi_1}\ar@{->}[rd]&
{\color{red}\xi_2}\ar@{<-}[d]&
{\color{red}\xi_3}\ar@{->}[d]&
{\color{red}\xi_4}\ar@{<-}[ld]\\
&x_1\ar@{->}[r]& x_2&
}\qquad
\stackrel{\mu_1}{\not\Longrightarrow}
\qquad
\xymatrix{
{\color{red}\xi_1}\ar@{<-}[rd]\ar@{->}[rrd]&
{\color{red}\xi_2}\ar@{->}[d]\ar@{<-}[rd]&
{\color{red}\xi_3}\ar@{->}[d]&
{\color{red}\xi_4}\ar@{<-}[ld]\\
&x'_1\ar@{<-}[r]& x_2&
}
$$
Indeed, this mutation would produce an extended quiver $\widetilde{\Qc}$
 such that 
$\widetilde{\Qc}_1$ contains exactly two $2$-paths through $x_2$, namely
$(\xi_1\to{}x_2\to\xi_2)$ and $(\xi_3\to{}x_2\to\xi_4)$.
Therefore in this case we have: $I_2=\{1,3\}\,,J_2=\{2,4\}$.
However, 
$\xi_1$ is not connected to $\xi_4$ and 
$\xi_3$ is not connected to $\xi_2$ by a $2$-path through $x_2$,
which violates Condition (C).

b)
The mutation of the following quiver at $x_1$ is not allowed
$$
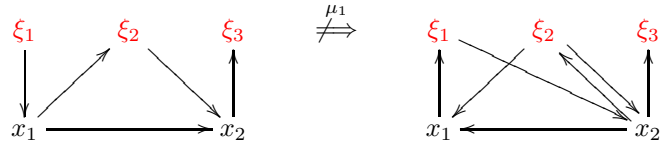
\xymatrix{
{\color{red}\xi_1}\ar@{->}[d]&
{\color{red}\xi_2}\ar@{<-}[ld]\ar@{->}[rd]&
{\color{red}\xi_3}\ar@{<-}[d]\\
x_1\ar@{->}[rr]&& x_2
}\qquad
\stackrel{\mu_1}{\not\Longrightarrow}
\qquad
\xymatrix{
{\color{red}\xi_1}\ar@{->}[rrd]\ar@{<-}[d]&
{\color{red}\xi_2}\ar@{->}[ld]\ar@<0.5ex>[rd]&
{\color{red}\xi_3}\ar@{<-}[d]\\
x_1\ar@{<-}[rr]&& x_2\ar@<0.5ex>[ul]
}
$$
since the subsets $I_2$ and $J_2$ are not well-defined.

c)
The following mutation is allowed
$$
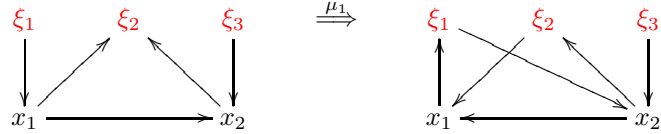
\xymatrix{
{\color{red}\xi_1}\ar@{->}[d]&
{\color{red}\xi_2}\ar@{<-}[ld]\ar@{<-}[rd]&
{\color{red}\xi_3}\ar@{->}[d]\\
x_1\ar@{->}[rr]&& x_2
}\qquad
\stackrel{\mu_1}{\Longrightarrow}
\qquad
\xymatrix{
{\color{red}\xi_1}\ar@{->}[rrd]\ar@{<-}[d]&
{\color{red}\xi_2}\ar@{->}[ld]\ar@{<-}[rd]&
{\color{red}\xi_3}\ar@{->}[d]\\
x_1\ar@{<-}[rr]&& x_2
}
$$
Indeed, in the resulting extended quiver, $I_2=\{1,3\}\,,J_2=\{2\}$,
and $\widetilde{\Qc}_1$ contains two $2$-paths trough $x_2$, namely
$(\xi_1\to{}x_2\to\xi_2)$ and $(\xi_3\to{}x_2\to\xi_2)$,
so that Condition (C) is satisfied.
\end{ex}

\subsection{Extended quivers with two colored vertices}

Consider an important class of extended quivers that have two 
colored vertices.
This is the first non-trivial example that already leads to
very interesting cluster superalgebras, see Section~\ref{ElExSec}.
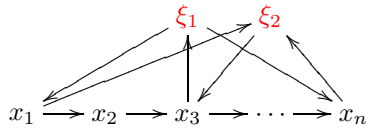
$$
 \xymatrix @!0 @R=1.3cm @C=1.1cm
 {
&&{\color{red}\xi_{1}}\ar@{->}[rrd]\ar@{->}[lld]\ar@{<-}[d]&
{\color{red}\xi_{2}}\ar@{<-}[rd]\ar@{<-}[llld]\ar@{->}[ld]&
&
\\
x_1\ar@{->}[r]&x_2\ar@{->}[r]&x_3\ar@{->}[r]&\cdots\ar@{->}[r]& x_n\\
}
$$
In the case of two colored vertices, 
{\it all mutations are allowed}.

\begin{ex}
For instance, one has
$$
 
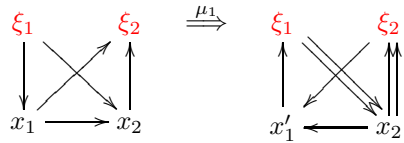
\xymatrix{
{\color{red}\xi_1}\ar@{->}[rd]\ar@{->}[d]&
{\color{red}\xi_2}\ar@{<-}[ld]\ar@{<-}[d]\\
x_1&x_2\ar@{<-}[l]
}
\quad
\stackrel{\mu_1}{\Longrightarrow}
\quad
 \xymatrix{
{\color{red}\xi_1}\ar@<3pt>@{->}[rd]\ar@{->}[rd]\ar@{<-}[d]&
{\color{red}\xi_2}\ar@{->}[ld]\ar@{<-}[d]\ar@<3pt>@{<-}[d]\\
x_1'&x_2\ar@{->}[l]
}
$$
This mutation creates a new $2$-path $(\xi_1\to{}x_2\to\xi_2)$,
so that the resulting extended quiver has two such paths.
\end{ex}

\section{Exchange relations and presymplectic form}\label{Echangist}

We define the exchange relations of the even variables
$\{x_1,\ldots,x_n\}$
corresponding to mutations of extended quivers defined in the previous section.
We introduce a presymplectic form invariant
with respect to mutations of the extended quiver and
the corresponding exchange relations.
We also formulate two maid results of the paper:
the Laurent phenomenon for the exchange relations,
and the invariance property of the presymplectic form. 

We finally define our notion of cluster superalgebra,
as the algebra generated by the variables $\{x_1,\ldots,x_n,\xi_1,\ldots,\xi_m\}$
called the  ``initial seed'',
and all their mutations $x_k',x_k'',\ldots$
defined by the exchange relations.
Note that every cluster superalgebra
has an {\it underlying} classical cluster algebra,
obtained by substituting $\xi_i\equiv0$.

\subsection{Exchange relations}

\begin{defn}
\label{ExchDef}
Given an extended quiver $\widetilde{\Qc}$,
the mutation $\mu_k$ replaces the variable $x_k$ by a new variable, $x_k'$,
other variables remain unchanged:
$$
\mu_k:\{x_1,\ldots,x_n,\xi_1,\ldots,\xi_m\}\to
\{x_1,\ldots,x_n,\xi_1,\ldots,\xi_m\}\setminus\{x_k\}\cup\{x_k'\}.
$$
The new variable is defined by the following formula
\begin{equation}
\label{Mute}
x_kx_k'=
\prod\limits_{\substack{x_k\to x_j }}x_j
\quad+\quad
\Big(1+
\sum\limits_{\substack{\xi_i\to{}x_k\to\xi_j}}\xi_i\xi_j\Big)
\prod\limits_{\substack{x_i\to x_k }}x_i,
\end{equation}
that will be called, as in the classical case, an exchange relation.
\end{defn}

Note that,
after substitution $\xi\equiv0$, the above formula obviously coincides with the
exchange relations for the classical cluster algebra corresponding to the underlying quiver $\Qc$.
The first summand in (\ref{Mute}) is exactly as in the classical case, the second
one is modified.

Our first main result is the Laurent phenomenon.
Let us emphacise that,
since the division by odd coordinates is not well-defined, 
all the Laurent polynomials we consider have denominators
equal to some monomials in $\{x_1,\ldots,x_n\}$.

\begin{thm}
\label{LeurThm}
For every extended quiver~$\widetilde{\Qc}$, all the 
rational functions $x_k', x_k'',\ldots$,
obtained recurrently by any series of consecutive admissible mutations,
 are Laurent polynomials
in the initial coordinates $\{x_1,\ldots,x_n,\xi_1,\ldots,\xi_m\}$.
\end{thm}

This theorem is proved in Section~\ref{PrLPSect}.

\begin{rem}
\label{SqMutLem}
Note that, unlike the purely even case, the above mutation of $x_k$ is
{\it not an involution},
i.e., if the function $x''_k$ is obtained by
the iterated mutation $\mu_k^2$ at $x_k$,
then $x''_k\not=x_k$.
Using~(\ref{Mute}) twice, one easily obtains:
$$
x''_k=x_k\,
\Big(1-
\sum\limits_{\substack{\xi_i\to{}x_k\to\xi_j}}\xi_i\xi_j\Big).
$$
It is important to notice however, that
this expression belongs to the algebra generated by
the initial variables $\{x_k,\xi_1,\ldots,\xi_m\}$.
Note also that, conversely,
$x_k=x''_k
\Big(1+
\sum\limits_{\substack{\xi_i\to{}x_k\to\xi_j}}\xi_i\xi_j\Big).$
\end{rem}

\subsection{The presymplectic form}
Consider the following differential $2$-form:
\begin{equation}
\label{SupOM}
\displaystyle
\om=
\sum\limits_{\substack{x_i\rightarrow x_j }}
\frac{dx_i\wedge{}dx_j}{x_ix_j}+
\sum\limits_{\substack{\xi_i\to{}x_\ell\to\xi_j}}
\frac{d\left(
\xi_i\xi_j
\right)\wedge
dx_\ell}{x_\ell}.
\end{equation}
Note that the summation is organized over the elements of $\widetilde\Qc_1$.
The first summand is nothing but the classical presymplectic form.

It turns out that the $2$-form~\eqref{SupOM} is invariant under mutations $\mu_k$.
The following statement is our second main result.

\begin{thm}
\label{SymThm}
For every extended quiver~$\widetilde{\Qc}$, the form $\om$
is invariant under mutations $\mu_k$, combined with the exchange relations~(\ref{Mute}).
\end{thm}
This theorem is proved in Section~\ref{ISFSect}

In other words, expressing $x_k$ in terms of the other variables and $x_k'$,
and substituting to~(\ref{SupOM}),
one obtains precisely the presymplectic form associated to the extended quiver 
$\mu_k(\widetilde{\Qc})$.
Geometrically speaking, the above statement means that the form~(\ref{SupOM})
is well-defined on an algebraic supermanifold whose space of regular functions
has a structure of cluster superalgebra.

Moreover, as in the purely even case,
formulas (\ref{Mute}) and (\ref{SupOM}) allow us to recover the rules of quiver mutations
from Definition~\ref{MutDef}.

\begin{rem}
The $2$-form $\om$ given by~(\ref{SupOM}) is not the unique invariant
presymplectic form.
Indeed, adding the (symmetric in $\xi$) terms of the form
$d\xi_i\wedge{}d\xi_j$, for some $i,j$, one obtains a large family
of presymplectic forms.
This is probably equivalent to developing of the theory
of mutations at colored vertices and exchange relations of odd variables.
Only adding such terms one can expect the $2$-form to be non-degenerate
when the underlying form is non-degenerate.
\end{rem}

\subsection{Definition of the cluster superalgebra}

Given an extended quiver~$\widetilde{\Qc}$,
we call the {\it cluster superalgebra} associated to~$\widetilde{\Qc}$,
the supercommutative superalgebra
generated by
the initial variables $\{x_1,\ldots,x_n,\xi_1,\ldots,\xi_m\}$, together with
the functions $x_k', x_k'',\ldots$ that can be obtained recursively by
sequences of admissible mutations.
This superalgebra will be denoted by $A(\widetilde{\Qc})$.

The pair $(\widetilde{\Qc},\{x_1,\ldots,x_n,\xi_1,\ldots,\xi_m\})$ is called 
the initial {\it seed} of the cluster superalgebra.

Given an extended quiver $\widetilde{\Qc}$, following~\cite{FZ1},
we call the {\it exchange graph} the graph~$\widetilde{T}$ whose vertices are 
the initial seed $(\widetilde{\Qc},\{x,\xi\})$ and all its allowed mutations;
the edges are labeled by the numbers from $1$ to $n$.
The graph~$\widetilde{T}$ has no natural orientation.
The projection $\widetilde{\Qc}\mapsto\Qc$ defines the projection $\widetilde{T}\mapsto{}T$,
where $T$ is the usual exchange graph.

\section{Elementary examples}\label{ElExSec}
In this section, we consider a number of concrete examples,
in order to illustrate the exchange relations~(\ref{Mute})
and the mutation rules of extended quivers.
Most of our examples correspond to extended quivers with two colored
vertices.
We observe a failure of $5$-periodicity that occurs in the
classical case.

We show that the Laurent phenomenon holds for all examples we consider,
and sometimes simplifications are quite spectacular.
We also check that the corresponding presymplectic form is, indeed, invariant
in each of our examples. 

\medskip

\paragraph{\bf A brief tutorial on calculus with
odd variables}
One needs to know a few simple rules in order to
work efficiently with odd variables.

Odd variables commute with
the even variables and anticommute with each other.
One has:
$
\xi_1\xi_2=-\xi_2\xi_1.
$
In particular, odd variables square to zero: $\xi_i^2=0$.

We will often need to calculate rational functions
with odd variables in the denominator.
Clearly, division by the odd variables is not well-defined.
However, the obvious formula
$(1+\xi)^{-1}=1-\xi$
makes many calculations possible.
For instance, 
$$
\frac{y}{x+\xi_1\xi_2}=\frac{y}{x(1+\xi_1\xi_2/x)}=\frac{y}{x}-\xi_1\xi_2\frac{y}{x^2}.
$$

\begin{ex}
\label{MirifSupMut1}
Our most elementary example is the extended quiver with one ordinary
and two colored vertices.
Let us first perform consecutive mutations at the even vertex:
$$
 
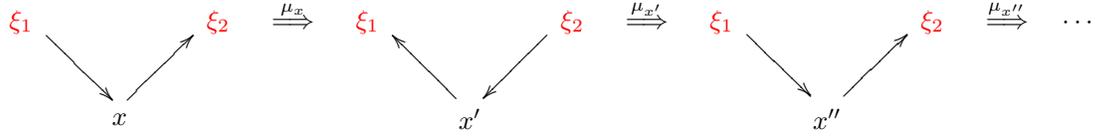
\xymatrix{
{\color{red}\xi_1}\ar@<3pt>@{->}[rd]&&
{\color{red}\xi_2}\ar@<-3pt>@{<-}[ld]\\
& x&
}
\quad
\stackrel{\mu_x}{\Longrightarrow}
\quad
 \xymatrix{
{\color{red}\xi_1}\ar@<3pt>@{<-}[rd]&&
{\color{red}\xi_2}\ar@<-3pt>@{->}[ld]\\
& x'&
}
\quad
\stackrel{\mu_{x'}}{\Longrightarrow}
\quad
 \xymatrix{
{\color{red}\xi_1}\ar@<3pt>@{->}[rd]&&
{\color{red}\xi_2}\ar@<-3pt>@{<-}[ld]\\
& x''&
}
\quad
\stackrel{\mu_{x''}}{\Longrightarrow}
\quad
\cdots
$$
The variable $x$ obeys the following exchange relation:
$$
x'=\frac{2}{x}+\frac{\xi_1\xi_2}{x},
\qquad
x''=
\frac{2}{x'}-\frac{\xi_1\xi_2}{x'}
=x\left(\frac{2}{2+\xi_1\xi_2}-\frac{\xi_1\xi_2}{2+\xi_1\xi_2}\right)
=x\left(1-\xi_1\xi_2\right).
$$
Continuing the process, we obtain:
$$
x'''=\frac{2}{x}+3\frac{\xi_1\xi_2}{x},
\qquad
x^{IV}=x\left(1-2\xi_1\xi_2\right),
\qquad
\ldots
$$
As already mentioned, these mutations are not involutions,
and the resulting process is infinite and aperiodic.

The corresponding cluster superalgebra is
$\C[x^{\pm1},\xi_1,\xi_2]\cong\C[x^{\pm1}]\otimes\C[\xi_1,\xi_2]$.

The presymplectic form corresponding to thisl extended quiver
is:
$$
\om=
\frac{d(\xi_1\xi_2)\wedge{}dx_1}{x_1}.
$$
Substituting $x=\frac{2}{x'}+\frac{\xi_1\xi_2}{x'}$,
and using the fact that $d(\xi_1\xi_2)\wedge{}d(\xi_1\xi_2)=0$,
one has:
$$
\om=
-\frac{d(\xi_1\xi_2)\wedge{}dx_1'}{x_1'},
$$
which corresponds to the reversing of the orientation of the $2$-path
$(\xi_1\to x\to\xi_2)$ by the mutation at $x$.
\end{ex}

\begin{ex}
\label{MirifSupMut0}
Recall (see~\cite{FZ1}) that, in the purely even case of the Dynkin quiver $\Qc=A_2$,
the process of consecutive mutations:
$$
\begin{array}{ccccccccc}
 \xymatrix{
x_1&x_2\ar@{<-}[l]
}
&
\stackrel{\mu_1}{\Longrightarrow}
&
 \xymatrix{
x_1'&x_2\ar@{->}[l]
}
&
\stackrel{\mu_2}{\Longrightarrow}
&
 \xymatrix{
x_1'&x_2'\ar@{<-}[l]
}
&
\stackrel{\mu_1}{\Longrightarrow}
&
 \xymatrix{
x_1''&x_2'\ar@{->}[l]
}
&
\stackrel{\mu_2}{\Longrightarrow}
&
\cdots
\end{array}
$$
 is $5$-periodic:
 $$
 x'_1=\frac{1+x_2}{x_1},
\qquad
x'_2=\frac{1+x_1+x_2}{x_1x_2},
\qquad
x''_1=\frac{1+x_1}{x_2},
\qquad
x''_2=x_1,
\qquad
x'''_1=x_2.
 $$
The corresponding exchange graph is a pentagon,
and the cluster algebra of $A_2$ is
isomorphic to $\C[x_1,x_2,x'_1,x'_2,x''_1].$
\end{ex}

In the next three examples, we consider different quivers that are extensions of the quiver $A_2$.

\begin{ex}
\label{MirifSupMut2}
Consider the following quiver with two even and two odd coordinates
$\{x_1,x_2,\xi_1,\xi_2\}$:
$$
\xymatrix{
{\color{red}\xi_1}\ar@{->}[d]&
{\color{red}\xi_2}\ar@<-1pt>@{<-}[ld]\\
x_1&x_2\ar@{<-}[l]
}
$$
Let us alternate consecutive mutations at even vertices:
$$
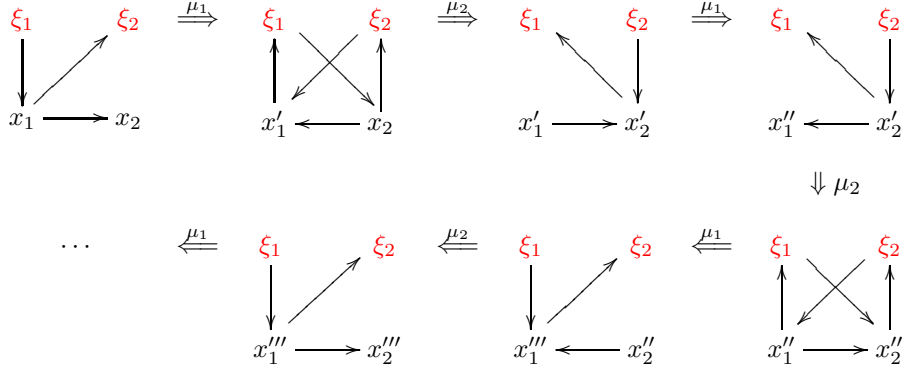
\begin{array}{ccccccc}
 \xymatrix{
{\color{red}\xi_1}\ar@{->}[d]&
{\color{red}\xi_2}\ar@<-1pt>@{<-}[ld]\\
x_1&x_2\ar@{<-}[l]
}
&
\stackrel{\mu_1}{\Longrightarrow}
&
 \xymatrix{
{\color{red}\xi_1}\ar@<3pt>@{->}[rd]\ar@{<-}[d]&
{\color{red}\xi_2}\ar@<-2pt>@{->}[ld]\ar@{<-}[d]\\
x_1'&x_2\ar@{->}[l]
}
&
\stackrel{\mu_2}{\Longrightarrow}
&
 \xymatrix{
{\color{red}\xi_1}\ar@<3pt>@{<-}[rd]&
{\color{red}\xi_2}\ar@{->}[d]\\
x_1'&x_2'\ar@{<-}[l]
}
&
\stackrel{\mu_1}{\Longrightarrow}
&
 \xymatrix{
{\color{red}\xi_1}\ar@<3pt>@{<-}[rd]&
{\color{red}\xi_2}\ar@{->}[d]\\
x_1''&x_2'\ar@{->}[l]
}
\\[50pt]
&&&&&&\Downarrow\mu_2\\[10pt]
\cdots
&\stackrel{\mu_1}{\Longleftarrow}&
\xymatrix{
{\color{red}\xi_1}\ar@{->}[d]&
{\color{red}\xi_2}\ar@<-3pt>@{<-}[ld]\\
x_1'''&x_2'''\ar@{<-}[l]
}
&\stackrel{\mu_2}{\Longleftarrow}& 
\xymatrix{
{\color{red}\xi_1}\ar@{->}[d]&
{\color{red}\xi_2}\ar@<-3pt>@{<-}[ld]\\
x_1'''&x_2''\ar@{->}[l]
}
&
\stackrel{\mu_1}{\Longleftarrow}
&
 \xymatrix{
{\color{red}\xi_1}\ar@<3pt>@{->}[rd]\ar@{<-}[d]&
{\color{red}\xi_2}\ar@<-3pt>@{->}[ld]\ar@{<-}[d]\\
x_1''&x_2''\ar@{<-}[l]
}
\end{array}
$$
Let us give some details of the computations.
The first two mutations are obtained immediately from~(\ref{Mute}):
$$
x'_1=\frac{1+x_2}{x_1}+\frac{\xi_1\xi_2}{x_1},
\qquad
x_2'=\frac{1+x_1+x_2}{x_1x_2}+\frac{1+x_1}{x_1x_2}\xi_1\xi_2.
$$
Again, applying the exchange relations~(\ref{Mute}), and performing simplifications,
one obtains:
$$
\begin{array}{rcl}
x_1''&=&
\frac{x_2'+1}{x_1'}\\[4pt]
&=&
\frac{1}{x_2}\left(
\frac{1+x_1+x_2+x_1x_2}{1+x_2+\xi_1\xi_2}
+\frac{\left(1+x_1\right)\xi_1\xi_2}{1+x_2+\xi_1\xi_2}
\right)\\[4pt]
&=&
\frac{1}{x_2}\left(
\frac{1+x_1+x_2+x_1x_2}{1+x_2}-
    \frac{1+x_1+x_2+x_1x_2}{1+x_2}\frac{\xi_1\xi_2}{1+x_2}
+\frac{1+x_1}{1+x_2}\xi_1\xi_2
\right)\\[6pt]
&=&\displaystyle
\frac{1+x_1}{x_2}.
\end{array}
$$
Note that simplifications of the denominators are quite surprising.

Similarly, after some computations, one obtains:
$$
x_2''=x_1\left(1-\xi_1\xi_2\right),
\qquad
x_1'''=x_2\left(1-\xi_1\xi_2\right),
\qquad
x_2'''=\frac{1+x_2}{x_1}+\frac{\xi_1\xi_2}{x_1},
$$
and the next values are:
$$
x_1^{IV}=
\frac{1+x_1+x_2}{x_1x_2}
\left(1+\xi_1\xi_2\right)
+\frac{1+x_1}{x_1x_2}\xi_1\xi_2,
\qquad
x_2^{IV}=\frac{1+x_1}{x_2}\left(1+\xi_1\xi_2\right),
\qquad
\ldots
$$
Remarkably enough, all of the coordinates we obtain by iterating the process,
are Laurent polynomials.
However, the above process is infinite and aperiodic.

Our mutation process is aperiodic and very different from the classical
mutation process of the quiver $A_2$, 
However, analyzing the above formulas, we see the following.
The corresponding cluster superalgebra is generated by five even and two odd variables:
$\{x_1,x_2,x'_1,x'_2,x''_1,\xi_1,\xi_2\}.$
Indeed, all the other mutations of~$x_1$ and~$x_2$ are
algebraic expressions in these functions.

\medskip

The presymplectic form corresponding to the initial extended quiver
$$
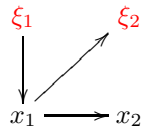
\xymatrix{
{\color{red}\xi_1}\ar@{->}[d]&
{\color{red}\xi_2}\ar@<-1pt>@{<-}[ld]\\
x_1&x_2\ar@{<-}[l]
}
$$
is as follows:
$$
\om=
\frac{dx_1\wedge{}dx_2}{x_1x_2}
+\frac{d(\xi_1\xi_2)\wedge{}dx_1}{x_1}.
$$
Let us substitute to this form
$x_1=\frac{1}{x_1'}\left(1+x_2+\xi_1\xi_2\right)$.
One readily gets:
$$
\begin{array}{rcl}
\om&=&\displaystyle
-\frac{dx_1'\wedge{}dx_2}{x_1'x_2}
-\frac{d(\xi_1\xi_2)\wedge{}dx_1'}{x_1'}
+\frac{d(\xi_1\xi_2)\wedge{}dx_2}{1+x_2+\xi_1\xi_2}
+\frac{d(\xi_1\xi_2)\wedge{}dx_2}{(1+x_2+\xi_1\xi_2)x_2}\\[12pt]
&=&\displaystyle
-\frac{dx_1'\wedge{}dx_2}{x_1'x_2}
-\frac{d(\xi_1\xi_2)\wedge{}dx_1'}{x_1'}
+\frac{d(\xi_1\xi_2)\wedge{}dx_2}{x_2}.
\end{array}
$$
The latter expression is precisely the
presymplectic form corresponding to
$$
 
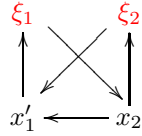
\xymatrix{
{\color{red}\xi_1}\ar@<3pt>@{->}[rd]\ar@{<-}[d]&
{\color{red}\xi_2}\ar@<-2pt>@{->}[ld]\ar@{<-}[d]\\
x_1'&x_2\ar@{->}[l]
}
$$
which is the result of the mutation $\mu_1$.
Hence, $\om$ is invariant under this mutation.
A similar computation shows that $\om$ is also invariant under the mutation $\mu_2$, etc. 
\end{ex}

\begin{ex}
\label{MirifSupMut3}
Starting with a similar quiver
and performing the same sequence of mutations, one has:
$$
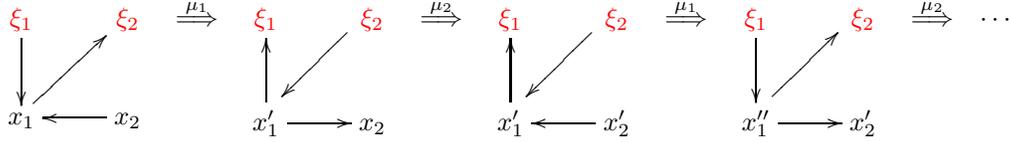
\begin{array}{ccccccccc}
 \xymatrix{
{\color{red}\xi_1}\ar@{->}[d]&
{\color{red}\xi_2}\ar@<-1pt>@{<-}[ld]\\
x_1&x_2\ar@{->}[l]
}
&\stackrel{\mu_1}{\Longrightarrow}&
 \xymatrix{
{\color{red}\xi_1}\ar@{<-}[d]&
{\color{red}\xi_2}\ar@<-3pt>@{->}[ld]\\
x'_1&x_2\ar@{<-}[l]
}
&\stackrel{\mu_2}{\Longrightarrow}&
 \xymatrix{
{\color{red}\xi_1}\ar@{<-}[d]&
{\color{red}\xi_2}\ar@<-3pt>@{->}[ld]\\
x'_1&x'_2\ar@{->}[l]
}
&\stackrel{\mu_1}{\Longrightarrow}&
 \xymatrix{
{\color{red}\xi_1}\ar@{->}[d]&
{\color{red}\xi_2}\ar@<-3pt>@{<-}[ld]\\
x''_1&x'_2\ar@{<-}[l]
}
&\stackrel{\mu_2}{\Longrightarrow}&\cdots
\end{array}
$$
and then again $\mu_1,\mu_2$, etc.
One then obtains, after computations:
$$
x_1'=\frac{1+x_2}{x_1}
+\frac{x_2}{x_1}\,\xi_1\xi_2,
\qquad
x_2'=\frac{1+x_1+x_2}{x_1x_2}+\frac{\xi_1\xi_2}{x_1},
\qquad
x_1''=\frac{1+x_1}{x_2}\left(1-\xi_1\xi_2\right),
$$
and next:
$$
x_2''=x_1\left(1-\xi_1\xi_2\right),
\qquad
x_1'''=x_2\left(1+\xi_1\xi_2\right),
\qquad
\ldots
$$
The Laurent phenomenon is always verified.

Note that, once again, the above process of consecutive mutations does not lead to $5$-periodicity.
However, once again, the corresponding cluster superalgebra is generated by the
following functions
$\{x_1,x_2,x'_1,x'_2,x''_1,\xi_1,\xi_2\}$.

\medskip

The presymplectic form in this example is:
$$
\om=
-\frac{dx_1\wedge{}dx_2}{x_1x_2}
+\frac{d(\xi_1\xi_2)\wedge{}dx_1}{x_1}.
$$
Substituting to this form
$x_1=\frac{1}{x_1'}\left(1+x_2(1+\xi_1\xi_2)\right)$, one obtains:
$$
\begin{array}{rcl}
\om&=&\displaystyle
\frac{dx_1'\wedge{}dx_2}{x_1'x_2}
-\frac{d(\xi_1\xi_2)\wedge{}dx_1'}{x_1'}
-\frac{x_2d(\xi_1\xi_2)\wedge{}dx_2}{(1+x_2+\xi_1\xi_2)x_2}
+\frac{d(\xi_1\xi_2)\wedge{}dx_2}{(1+x_2+\xi_1\xi_2)x_2}\\[12pt]
&=&\displaystyle
\frac{dx_1'\wedge{}dx_2}{x_1'x_2}
-\frac{d(\xi_1\xi_2)\wedge{}dx_1'}{x_1'}.
\end{array}
$$
Since this expression is nothing but the presymplectic form corresponding to
the following extended quiver:
$$
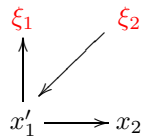
\xymatrix{
{\color{red}\xi_1}\ar@{<-}[d]&
{\color{red}\xi_2}\ar@<-3pt>@{->}[ld]\\
x'_1&x_2\ar@{<-}[l]
}
$$
We conclude, that the $2$-form $\om$ is, again, invariant 
with respect to the mutation $\mu_1$.
It is not difficult to check that $\om$ is invariant
with respect to all the other mutations
we consider in this example.
\end{ex}

\begin{ex}
\label{MirifSupMut4}
The following extended quiver
with two even and three odd vertices, 
as well as its mutations, will be particularly important for us
(see Section~\ref{SFS} below).
As before, let us perform an infinite sequence of mutations:
$$
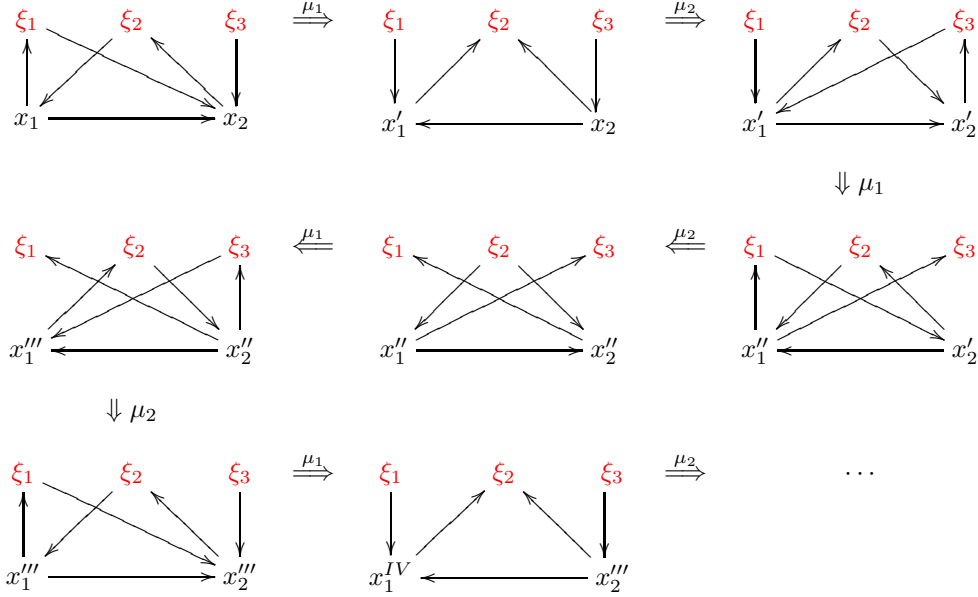
\begin{array}{ccccc}
 \xymatrix{
{\color{red}\xi_1}\ar@{<-}[d]\ar@{->}[rrd]&
{\color{red}\xi_2}\ar@{->}[ld]\ar@{<-}[rd]&
{\color{red}\xi_3}\ar@{->}[d]\\
x_1\ar@{->}[rr]&&x_2
}
&
\stackrel{\mu_1}{\Longrightarrow}
&
 \xymatrix{
{\color{red}\xi_1}\ar@{->}[d]&
{\color{red}\xi_2}\ar@{<-}[ld]\ar@{<-}[rd]&
{\color{red}\xi_3}\ar@<-3pt>@{->}[d]\\
x_1'\ar@{<-}[rr]&&x_2
}
&
\stackrel{\mu_2}{\Longrightarrow}
&
\xymatrix{
{\color{red}\xi_1}\ar@{->}[d]&
{\color{red}\xi_2}\ar@{<-}[ld]\ar@{->}[rd]&
{\color{red}\xi_3}\ar@{->}[lld]\ar@{<-}[d]\\
x_1'\ar@{->}[rr]&&x_2'
}
\\[50pt]
&&&&\Downarrow\mu_1\\[10pt]
\xymatrix{
{\color{red}\xi_1}\ar@{<-}[rrd]&
{\color{red}\xi_2}\ar@{<-}[ld]\ar@{->}[rd]&
{\color{red}\xi_3}\ar@{->}[lld]\ar@{<-}[d]\\
x'''_1\ar@{<-}[rr]&&x''_2
}
&\stackrel{\mu_1}{\Longleftarrow}&
\xymatrix{
{\color{red}\xi_1}\ar@{<-}[rrd]&
{\color{red}\xi_2}\ar@{->}[ld]\ar@{->}[rd]&
{\color{red}\xi_3}\ar@{<-}[lld]\\
x''_1\ar@{->}[rr]&&x''_2
}
&
\stackrel{\mu_2}{\Longleftarrow}
&
\xymatrix{
{\color{red}\xi_1}\ar@{<-}[d]\ar@{->}[rrd]&
{\color{red}\xi_2}\ar@{->}[ld]\ar@{<-}[rd]&
{\color{red}\xi_3}\ar@{<-}[lld]\\
x''_1\ar@{<-}[rr]&&x'_2
}
\\[50pt]
\Downarrow\mu_2\\[10pt]
\xymatrix{
{\color{red}\xi_1}\ar@{<-}[d]\ar@{->}[rrd]&
{\color{red}\xi_2}\ar@{->}[ld]\ar@{<-}[rd]&
{\color{red}\xi_3}\ar@{->}[d]\\
x'''_1\ar@{->}[rr]&&x'''_2
}
&
\stackrel{\mu_1}{\Longrightarrow}
&
 \xymatrix{
{\color{red}\xi_1}\ar@{->}[d]&
{\color{red}\xi_2}\ar@{<-}[ld]\ar@{<-}[rd]&
{\color{red}\xi_3}\ar@<-3pt>@{->}[d]\\
x^{IV}_1\ar@{<-}[rr]&&x'''_2
}
&
\stackrel{\mu_2}{\Longrightarrow}
&\cdots
\end{array}
$$
One then has: 
$$
\begin{array}{ll}
\displaystyle
x_1'=\frac{1+x_2}{x_1}-\frac{\xi_1\xi_2}{x_1},
&
\displaystyle
x_2'=
\frac{1+x_1+x_2}{x_1x_2}
-\frac{\xi_1\xi_2}{x_1x_2}-\frac{\xi_2\xi_3}{x_2},
\\[10pt]
\displaystyle
x_1''=\frac{1+x_1}{x_2}
+\frac{x_1}{x_2}\left(\xi_1+\xi_3\right)\xi_2,
&
\displaystyle
x_2''=x_1\left(1+\xi_1\xi_2\right)\\[10pt]
\displaystyle
x_1'''=x_2\left(1+\xi_2\xi_3\right),&
\end{array}
$$
and so on.
The next mutations of the quiver
and of the coordinates can be easily calculated,
and once again, the Laurent phenomenon is verified.

\medskip

The presymplectic form in this example is as follows:
$$
\om=\frac{dx_1\wedge{}dx_2}{x_1x_2}
-\frac{d(\xi_1\xi_2)\wedge{}dx_1}{x_1}
+\frac{d\left(\xi_1\xi_2\right)\wedge{}dx_2}{x_2}
-\frac{d\left(\xi_2\xi_3\right)\wedge{}dx_2}{x_2}.
$$
Although the computation in this case is slightly more involved,
one readily checks its invariance under the above mutations.
\end{ex}

\begin{ex}
{\bf The supergroup $\OSp(1|2)$}.
One of the first examples of cluster algebras
given in~\cite{FZ1} is the algebra of regular functions on the Lie group $\SSL(2)$.
We consider here its superanalog.

The most elementary superanalog of the group $\SSL(2)$ is the supergroup $\OSp(1|2)$.
For more details about properties and applications of this supergroup, see~\cite{Man}.
Let $\Rc=\Rc_0\oplus\Rc_1$ be a commutative ring.
The set of $\Rc$-points of the supergroup $\OSp(1|2)$  is the following $3|2$-dimensional supergroup of matrices:
\begin{equation}
\label{OSpRel}
\left(
\begin{array}{cc|c}
a&b&\g\\[4pt]
c&d&\d\\[4pt]
\hline
\a&\b&e
\end{array}
\right)
\qquad
\hbox{such that}
\qquad
\begin{array}{rcl}
ad&=&1+bc-\a\b,\\[4pt]
e&=&1+\a\b,\\[4pt]
\g&=&a\b-b\a\\[4pt]
\d&=&c\b-d\a.
\end{array}
\end{equation}
The elements $a,b,c,d,e\in\Rc_0$, and $\a,\b,\g,\d\in\Rc_1$;
these elements are generators of the algebra of regular functions on $\OSp(1|2)$.

Choose the initial cluster coordinates $(a,b,c,\a,\b)$,
and consider the following quiver:
$$
 
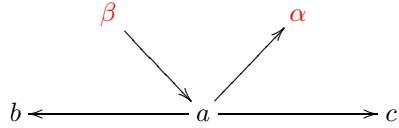
\xymatrix{
&{\color{red}\b}\ar@{->}[rd]&&
{\color{red}\a}\ar@{<-}[ld]\\
b\ar@{<-}[rr]&& a\ar@{->}[rr]&&c
}
$$
The coordinate $d$ is then the mutation of $a$, i.e., $a'=d$.
Indeed, the exchange relation~\eqref{Mute} for the coordinate $a$ reads
$$
aa'=1+bc+\b\a,
$$
which is precisely the first equation for $\OSp(1|2)$ relating $a$ and $d$.
Note that, similarly to the $\SL_2$-case 
the coordinates $b$ and $c$ are frozen cf.~\cite{FZ1}.
\end{ex}

A challenging problem is to find more geometric examples of of cluster superalgebras:
Lie supergroups, Grassmannians, etc.

\section{Superfriezes and cluster superalgebras of type $A$}\label{SFS}

Frieze patterns were invented by Coxeter~\cite{Cox}
(see also~\cite{CoCo}).
This notion provides surprising relations between classical continued fractions,
projective geometry (cross-ratios) and quiver representations.
Coxeter's friezes are also related to linear difference equations and
the classical moduli spaces $\mathcal{M}_{0,n}$ of configurations of points,
see~\cite{SVRS}.
The set of Coxeter's friezes is an algebraic variety
that has a structure of cluster algebras,
associated to the Dynkin quivers $A_n$.
For a modern survey, see~\cite{Mor}.

The notion of {\it superfrieze} was introduced in~\cite{SFriZ}
as generalization of Coxeter's frieze patterns.
The collection of all superfriezes is an algebraic supervariety
isomorphic to the supervariety of supersymmetric Hill's 
(or one-dimensional Schr\"odinger) equations~\cite{SVRS}
with some particular monodromy condition.

In this section, we describe the structure of cluster superalgebra
on the supervariety of superfriezes.
The underlying quiver will be $A_n$, the extended quiver~$\widetilde{\Qc}$
is obtained by adding $n+1$ odd coordinates.
We start with the description of the supergroup $\OSp(1|2)$
that is used to develop the theory of superfriezes.

\subsection{Supersymmetric discrete Schr\"odinger equation}
Consider two infinite sequences of elements of some supercommutative ring $\Rc$:
$$
(a_i),
\quad
(\b_i),
\quad
i\in\Z,
$$
where $a_i\in\Rc_0$ and $\b_i\in\Rc_1$.

The following equation with indeterminate $(V_i,W_i)_{i\in\Z}$:
\begin{equation}
\label{SeQE}
\left(
\begin{array}{l}
V_{i-1}\\[4pt]
V_i\\[4pt]
W_i
\end{array}
\right)=
A_i\left(
\begin{array}{l}
V_{i-2}\\[4pt]
V_{i-1}\\[4pt]
W_{i-1}
\end{array}
\right),
\qquad
\hbox{where}
\qquad
A_i=
\left(
\begin{array}{cc|c}
0&1&0\\[4pt]
-1&a_i&-\b_i\\[4pt]
\hline
0&\b_i&1
\end{array}
\right),
\end{equation}
is the supersymmetric version of discrete Schr\"odinger equation, see~\cite{SFriZ}.
Note that the matrix $A_i$ belongs to the supergroup $\OSp(1|2)$.

We assume that the coefficients $a_i,\b_i$ are 
(anti)periodic with some period $n$:
$$
a_{i+n}=a_i,
\qquad
\b_{i+n}=-b_i,
$$
for all $i\in\Z$.
Under this assumption, there is a notion of {\it monodromy},
i.e., an element $M\in\OSp(1|2)$, such that periodicity properties of
the solutions of~(\ref{SeQE}) are described by $M$.

Supersymmetric discrete Schr\"odinger equations
with fixed monodromy matrix:
\begin{equation}
\label{MoQE}
M=\left(
\begin{array}{rr|c}
-1&0&\;\;0\\[4pt]
0&-1&\;\;0\\[4pt]
\hline
0&0&\;\;1
\end{array}
\right).
\end{equation}
considered in~\cite{SFriZ}.
This is an algebraic supervariety of dimension $n|(n+1)$
which is a version of super moduli space $\mathfrak{M}_{0,n}$, see~\cite{Wit}.
The notion of superfrieze allows one to define special coordinates on this supervariety.

\subsection{The definition of a superfrieze and the corresponding superalgebra}\label{TheDef}

Similarly to the case of classical Coxeter's friezes, a superfrieze
is a horizontally-infinite array bounded by rows of $0$'s and $1$'s.
Even and odd elements alternate and form ``elementary diamonds'';
there are twice more odd elements.

\begin{defn}
A superfrieze, or a supersymmetric frieze pattern, is the following array
$$
\begin{array}{ccccccccccccccccccccccccc}
&\ldots&0&&&&0&&&&0\\[10pt]
\ldots&{\color{red}0}&&{\color{red}0}&&{\color{red}0}
&&{\color{red}0}&&{\color{red}0}&&\ldots\\[10pt]
\;\;\;1&&&&1&&&&1&&&\ldots\\[10pt]
&{\color{red}\varphi_{0,0}}&&{\color{red}\varphi_{\frac{1}{2},\frac{1}{2}}}&&{\color{red}\varphi_{1,1}}
&&{\color{red}\varphi_{\frac{3}{2},\frac{3}{2}}}&&{\color{red}\varphi_{2,2}}&&\ldots\\[12pt]
&&f_{0,0}&&&&f_{1,1}&&&&f_{2,2}\\[10pt]
&{\color{red}\varphi_{-\frac{1}{2},\frac{1}{2}}}&&{\color{red}\varphi_{0,1}}
&&{\color{red}\varphi_{\frac{1}{2},\frac{3}{2}}}
&&{\color{red}\varphi_{1,2}}&&{\color{red}\varphi_{\frac{3}{2},\frac{5}{2}}}&&\ldots\\[10pt]
f_{-1,0}&&&&f_{0,1}&&&&f_{1,2}&&\\[4pt]
&\iddots&&\iddots&& \ddots&&\ddots&& \ddots&&\!\!\!\ddots\\[4pt]
&&f_{2-m,1}&&&&f_{0,m-1}&&&&f_{1,m}&&&&\\[10pt]
\ldots&{\color{red}\varphi_{\frac{3}{2}-m,\frac{3}{2}}}&&{\color{red}\varphi_{2-m,2}}&&\ldots
&&{\color{red}\varphi_{0,m}}&&{\color{red}\varphi_{\frac{1}{2},m+\frac{1}{2}}}&&{\color{red}\varphi_{1,m+1}}\\[10pt]
\;\;\;1&&&&1&&&&1&&&&&\\[10pt]
\ldots&{\color{red}0}&&{\color{red}0}&&{\color{red}0}
&&{\color{red}0}&&{\color{red}0}&&{\color{red}0}&\\[10pt]
&\ldots&0&&&&0&&&&0&\ldots
\end{array}
$$
where $f_{i,j}\in\Rc_0$ and $\varphi_{i,j}\in\Rc_1$, and where every 
{\it elementary diamond}:
$$
\begin{array}{ccccc}
&&B&&\\[4pt]
&{\color{red}\Xi}&&{\color{red}\Psi}&\\[4pt]
A&&&&D\\[4pt]
&{\color{red}\Phi}&&{\color{red}\Sigma}&\\[4pt]
&&C&&
\end{array}
$$
satisfies the following conditions:
\begin{equation}
\label{Rule}
\begin{array}{rcl}
AD-BC&=&1+\Sigma\Xi,\\[4pt]
B\Phi-A\Psi&=&\Xi,\\[4pt]
B\Sigma-D\Xi&=&\Psi,
\end{array}
\end{equation}
that we call the {\it frieze rule}.

The integer $m$, i.e., the number of even rows between the rows
of $1$'s is called the {\it width} of the superfrieze.
\end{defn}

The last two equations of~(\ref{Rule}) are equivalent to
$$
A\Sigma-C\Xi=\Phi,
\qquad
D\Phi-C\Psi=\Sigma.
$$
Note also that these equations also imply $\Xi\Sigma=\Phi\Psi$,
so that the first equation can also be written as follows: 
$AD-BC=1-\Xi\Sigma$.

One can associate an elementary diamond with every element of $\OSp(1|2)$
using the following formula:
$$
\left(
\begin{array}{cc|c}
a&b&\g\\[4pt]
c&d&\d\\[4pt]
\hline
\a&\b&e
\end{array}
\right)
\qquad
\longleftrightarrow
\qquad
\begin{array}{ccccc}
&&\!\!\!-a&&\\[4pt]
&{\color{red}\g}&&{\color{red}\a}&\\[4pt]
b&&&&\!\!\!\!-c\\[4pt]
&\!\!\!{\color{red}-\b}&&{\color{red}\d}&\\[4pt]
&&d&&
\end{array}
$$
so that the relations~(\ref{OSpRel}) and~(\ref{Rule}) coincide.

Consider also the configuration:
$$
\begin{array}{ccccccc}
&&{\color{red}\widetilde{\Psi}}&&{\color{red}\widetilde{\Xi}}\\[4pt]
&&&B&&\\[4pt]
{\color{red}\widetilde{\Phi}}&&{\color{red}\Xi}&&{\color{red}\Psi}&&{\color{red}\widetilde{\Sigma}}\\[4pt]
&A&&&&D\\[4pt]
&&{\color{red}\Phi}&&{\color{red}\Sigma}&\\[4pt]
&&&C&&
\end{array}
$$
The frieze rule~\eqref{Rule} then implies
$$
B\,(\Phi-\widetilde{\Phi})=
A\,(\Psi-\widetilde{\Psi}),
\qquad
B\,(\Sigma-\widetilde{\Sigma})=
D\,(\Xi-\widetilde{\Xi}).
$$

\begin{defn}
The supercommutative superalgebra generated by
all the entries of a superfrieze will
be called the {\it algebra of a superfrieze}.
\end{defn}

\subsection{Examples: superfriezes of width $1$ and $2$}

The most general superfrieze of width $m=1$ is of the following form:
$$
\begin{array}{ccccccccccccccccccccccc}
&&0&&&&0&&&&\!\!0&&&&0\\[8pt]
&{\color{red}0}&&{\color{red}0}&&{\color{red}0}&&
\!\!\!{\color{red}0}&&{\color{red}0}&&{\color{red}0}&&\;\;\;{\color{red}0}&&\!\!\!{\color{red}0}\\[8pt]
1&&&&1&&&&1&&&&\!\!1&&&&1\\[8pt]
&{\color{red}\xi}&&{\color{red}\xi}&&{\color{red}\xi'}
&&{\color{red}\xi'}&&{\color{red}\xi-x\eta}&&{\color{red}\xi-x\eta}
&&{\color{red}\eta}&&{\color{red}\eta}\\[10pt]
&&x&&&&x'&&&&\!\!\!x&&&&x'\\[10pt]
&{\color{red}\xi-x\eta}&&\;\;\;{\color{red}x\eta-\xi}&&\;{\color{red}\eta}&&{\color{red}-\eta}
&&{\color{red}-\xi}&&{\color{red}\xi}&&{\color{red}-\xi'}&&{\color{red}\xi'}\\[8pt]
1&&&&1&&&&1&&&&\!\!1&&&&1\\[8pt]
&{\color{red}0}&&{\color{red}0}&&\;\;{\color{red}0}
&&\!{\color{red}0}&&{\color{red}0}&&{\color{red}0}&&\;\;\;{\color{red}0}&&\!\!\!{\color{red}0}\\[10pt]
&&0&&&&0&&&&\!\!0&&&&0\
\end{array}
$$
where
$$
x'=\frac{2}{x}+\frac{\eta\xi}{x},
\qquad 
\xi'=\eta-\frac{2\xi}{x}.
$$
One can chose local coordinates $(x,\xi,\eta)$
to parametrize the supervariety of superfriezes.
The value of $x'$ precisely corresponds to the mutation of~$x$
in Example~\ref{MirifSupMut1}.

The next example is a superanalog of
so-called Gauss' ``Pentagramma mirificum'':
$$
\begin{array}{cccccccccccccccccccccc}
0&&&&\!\!\!0&&&&\!\!0&&&&\!\!\!0&&&&0&&\ldots\\[10pt]
&{\color{red}0}&&\!\!{\color{red}0}&&{\color{red}0}
&&{\color{red}0}&&{\color{red}0}&&{\color{red}0}&&\!\!{\color{red}0}&&{\color{red}0}&&{\color{red}0}\\[10pt]
\ldots&&1&&&&1&&&&1&&&&\!\!1&&&&\!\!1\\[4pt]
&{\color{red}\xi^*}
&&\!\!{\color{red}\xi}&&{\color{red}\xi}&&{\color{red}\xi'}&&{\color{red}\xi'}
&&{\color{red}\raisebox{.5pt}{\textcircled{\raisebox{-.9pt} {1}}}}
&&{\color{red}\raisebox{.5pt}{\textcircled{\raisebox{-.9pt} {1}}}}
&&\!\!{\color{red}\zeta^*}&&{\color{red}\zeta^*}
\\[10pt]
y'&&&&
\!\!x&&&&\!\!x'
&&&&x''&&&&y\\[10pt]
&\!\!\!{\color{red}-\eta'}&&
\!\!\!{\color{red}\eta^*}
&&{\color{red}\raisebox{.5pt}{\textcircled{\raisebox{-.9pt} {2}}}}
&&{\color{red}\eta}
&&{\color{red}\raisebox{.5pt}{\textcircled{\raisebox{-.9pt} {2}}}'}
&&{\color{red}\eta'}
&&\!\!\!{\color{red}\eta^*}&&\!\!{\color{red}-\eta}
&&{\color{red}\eta}\\[10pt]
&&x''&&&&y
&&&&y'&&&&x
&&&&\!\!\!x'\\[10pt]
&{\color{red}\raisebox{.5pt}{\textcircled{\raisebox{-.9pt} {1}}}}
&&{\color{red}-\raisebox{.5pt}{\textcircled{\raisebox{-.9pt} {1}}}}
&&{\color{red}\zeta^*}
&&{\color{red}-\zeta^*}
&&{\color{red}\zeta}
&&\!\!{\color{red}-\zeta}&&\!\!\!{\color{red}\zeta'}
&&{\color{red}-\zeta'}&&{\color{red}-\xi'}\\[10pt]
1&&&&\!\!1&&&&1&&&&1&&&&1&&\ldots\\[10pt]
&{\color{red}0}&&\!\!{\color{red}0}&&{\color{red}0}&&{\color{red}0}&&{\color{red}0}
&&{\color{red}0}&&\!\!{\color{red}0}&&{\color{red}0}&&{\color{red}0}\\[10pt]
\ldots&&0&&&&0&&&&0&&&&0&&&&\!\!0
\end{array}
$$
The frieze is defined by the initial values
$(x,y,\xi,\eta,\zeta)$, the next values are easily calculated using the frieze rule:
$$
x'=\frac{1+y}{x}+\frac{\eta\xi}{x},
\qquad
y'=\frac{1+x+y}{xy}+\frac{\eta\xi}{xy}+\frac{\zeta\eta}{y}.
$$
Note that these formulas coincide with those of mutations of coordinates $x_1,x_2$
in Example~\ref{MirifSupMut4}.
One then calculates:
$$
x''=\frac{1+y'}{x'}+\frac{\eta'\xi'}{x'}
=\frac{1+x}{y}+\frac{\eta\xi}{y}+\xi\zeta+\frac{x}{y}\zeta\eta.
$$
All these Laurent polynomials
can be obtained as mutations of the
initial coordinates $(x,y,\xi,\eta,\zeta)$ and the initial quiver
$$
 
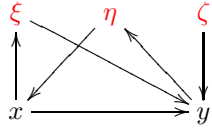
\xymatrix{
{\color{red}\xi}\ar@{<-}[d]\ar@{->}[rrd]&
{\color{red}\eta}\ar@{->}[ld]\ar@{<-}[rd]&
{\color{red}\zeta}\ar@{->}[d]\\
x\ar@{->}[rr]&&y
}
$$
see Example~\ref{MirifSupMut4}.

For the odd coordinates, one has:
$$
\xi'=\eta-x'\xi=\eta-\frac{1+y}{x}\xi,
\qquad
\eta'=\zeta-y'\xi=\zeta-\frac{1+x+y}{xy}\xi-\frac{\xi\eta\zeta}{y},
\qquad
\zeta'=-\xi
$$
On the other side of the initial diagonal,
$$
\zeta^*=\eta-y\zeta,
\qquad
\eta^*=\xi-x\zeta,
\qquad
\xi^*=-\zeta.
$$
Furthermore,
$$
\raisebox{.5pt}{\textcircled{\raisebox{-.9pt} {1}}}=
\frac{(1+x)}{y}\eta-\xi-\zeta,
\qquad
\raisebox{.5pt}{\textcircled{\raisebox{-.9pt} {2}}}=
x\eta-y\xi,
$$
and finally:
$$
\raisebox{.5pt}{\textcircled{\raisebox{-.9pt} {2}}}'=
x'\zeta-y'\eta
=\frac{1+y}{x}\zeta-\frac{1+x+y}{xy}\eta-\frac{\xi\eta\zeta}{x}.
$$

\subsection{Properties of superfriezes}

The main properties of superfriezes are similar to those of the classical Coxeter friezes, see~\cite{SFriZ}.

(a)
The property of {\it glide symmetry} reads:
$$
f_{i,j}=f_{j-m-1,i-2},\qquad
\varphi_{i,j}=\varphi_{j-m-\frac{3}{2},i-\frac32},\qquad
\varphi_{i+\frac12,j+\frac12}=-\varphi_{j-m-1,i-1}.
$$
This implies, in particular, the following (anti)periodicity:
$$
\varphi_{i+n,j+n}=-\varphi_{i,j},
\qquad
f_{i+n,j+n}=f_{i,j},
$$ 
for all $i,j\in\Z$.

(b)
The Laurent phenomenon:
entries of a superfrieze are Laurent polynomials in 
the entries from any of its diagonals.

(c)
The collection of all superfriezes of width $n$ is an algebraic supervariety
of superdimension~$n|(n+1)$.
It is isomorphic to the supervariety of
Schr\"odinger equations~(\ref{SeQE}) with monodromy condition~(\ref{MoQE}).
The relation to difference equations is as follows.
The entries of the South-East diagonal of every superfrieze are solutions to
the discrete Schr\"odinger equation~(\ref{SeQE}).

\subsection{Superfriezes and cluster superalgebras}

Let us now describe the cluster structure of the supervariety of superfriezes.
Consider the following quiver with $m$ even and $m+1$ odd vertices:
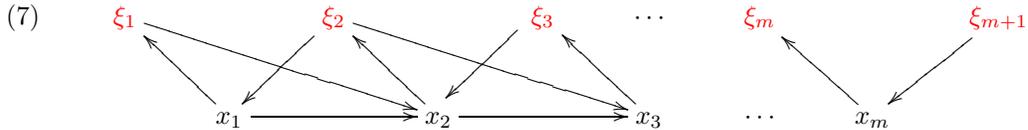
\begin{equation}
\label{AQuiv}
 \xymatrix{
{\color{red}\xi_1}\ar@{<-}[rd]\ar@{->}[rrrd]&&
{\color{red}\xi_2}\ar@{->}[ld]\ar@{<-}[rd]\ar@{->}[rrrd]&&
{\color{red}\xi_3}\ar@<-3pt>@{->}[ld]\ar@{<-}[rd]
&\cdots&
{\color{red}\xi_{m}}&&
{\color{red}\xi_{m+1}}\ar@{->}[ld]\\
& x_1\ar@{->}[rr]&&
x_2\ar@{->}[rr]
&&x_3
&\cdots&\ar@{->}[lu]x_m
}
\end{equation}
and the corresponding cluster superalgebra.

\begin{thm}
\label{ClFrProp}
The algebra of a superfrieze of width $m$ 
is a subalgebra of the cluster superalgebra corresponding to the above quiver.
\end{thm}

\begin{proof}
Choose the following entries of the superfrieze on parallel diagonals:
$$
 \begin{array}{cccccccccccccc}
1&&&&1\\[4pt]
&{\color{red}*}&&{\color{red}\xi_1}&&{\color{red}*}&&{\color{red}\xi'_1}\\[4pt]
&&x_1&&&&x'_1\\[4pt]
&&&{\color{red}*}&&{\color{red}\xi_2}&&{\color{red}*}&&{\color{red}\xi'_2}\\[4pt]
&&&&x_2&&{\color{red}\ddots}&&x'_2&&\!\!{\color{red}\ddots}\\[4pt]
&&&&&\ddots&&{\color{red}\xi_m}&&\ddots&&{\color{red}\xi'_m}\\[4pt]
&&&&&&x_m&&&&x'_m\\[4pt]
&&&&&&&{\color{red}*}&&{\color{red}\xi_{m+1}}&&{\color{red}*}&&{\color{red}\xi'_{m+1}}\\[4pt]
&&&&&&&&1&&&&1
\end{array}
$$
The entries $\{x_1,\ldots,x_m,\xi_1,\ldots,\xi_{m+1}\}$ determine all other entries of the superfrieze, 
and can be taken for initial coordinates.
Our goal is to calculate the entries $\{x'_1,\ldots,x'_m,\xi'_1,\ldots,\xi'_{m+1}\}$
and show that these entries 
also belong to the cluster superalgebra $A(\widetilde\Qc)$ of the quiver~\eqref{AQuiv}.

Using the frieze rule~\eqref{Rule}, one obtains the following recurrent formula:
\begin{equation}
\label{LFor}
x_kx'_k=1+x_{k+1}x'_{k-1}+\xi_{k+1}\xi_k.
\end{equation}
On the other hand, let us perform consecutive mutations at vertices 
$x_1$, and then at $x_2,x_3\ldots,x_m$
of the quiver~\eqref{AQuiv}\footnote{
Note that, according to Lemma~\ref{AMProp} below, this is the only allowed sequence of even mutations.}.
After the $(k-1)$st step, one obtains the following quiver:
$$
 \xymatrix
  @!0 @R=1.3cm @C=1.3cm
  {
{\color{red}\xi_1}\ar@{->}[rd]&&
{\color{red}\xi_2}\ar@{<-}[ld]\ar@{->}[rd]&&
{\color{red}\xi_3}\ar@<-1pt>@{<-}[ld]\ar@{->}[llld]
&\cdots&
{\color{red}\xi_{k}}\ar@{<-}[ld]\ar@{->}[rrrd]&&
{\color{red}\xi_{k+1}}\ar@{->}[ld]\ar@{<-}[rd]&\cdots\\
& x'_1\ar@{->}[rr]&&
x'_2
&\cdots&x'_{k-1}&&
\ar@{->}[lu]x_k\ar@{->}[rr]\ar@{->}[ll]&&x_{k+1}&\cdots
}
$$
Therefore, the mutation at $x_k$ is allowed,
and the exchange relation for $x_k$ is exactly the same as the
recurrent formula \eqref{LFor} for $x'_k$.
We have proved that the values of the entries $\{x'_1,\ldots,x'_m\}$ in the frieze coincide
with the coordinates $\{x'_1,\ldots,x'_m\}$ of the quiver~\eqref{AQuiv} after 
the iteration of even mutations.

Note that after $m$ consecutive mutations at even vertices, the quiver~\eqref{AQuiv}
becomes as follows:
$$
 \xymatrix
   @!0 @R=1.3cm @C=1.3cm
 {
{\color{red}\xi_1}\ar@{->}[rd]&&
{\color{red}\xi_2}\ar@{<-}[ld]\ar@{->}[rd]&&
{\color{red}\xi_3}\ar@<-1pt>@{<-}[ld]\ar@{->}[rd]\ar@{->}[llld]
&&{\color{red}\xi_3}\ar@<-1pt>@{<-}[ld]\ar@{->}[llld]&\cdots&
{\color{red}\xi_{m}}\ar@{<-}[ld]&&
{\color{red}\xi_{m+1}}\ar@{<-}[ld]\ar@{->}[llld]\\
& x'_1\ar@{->}[rr]&&
x'_2\ar@{->}[rr]
&&x'_3
&\cdots&x'_{m-1}\ar@{->}[rr]&&\ar@{<-}[lu]x'_m
}
$$

Consider now for the odd entries of the superfrieze $\{\xi'_1,\ldots,\xi'_{m+1}\}$,
and let us proceed by induction.

For the first of the odd entries, one has:
$$
\xi'_1=\xi_2-x'_1\xi_1.
$$
Indeed, the frieze rule implies that the entry between 
$\xi_1$ and $\xi'_1$ (previously denoted by $*$) is also equal to $\xi'_1$,
i.e., we have the following fragment of the superfrieze:
$$
\begin{array}{cccccccccc}
&&&&1&&&&1\\[4pt]
&&&{\color{red}\xi_1}&&{\color{red}\xi'_1}&&{\color{red}\xi'_1}\\[4pt]
&&x_1&&&&x'_1\\[4pt]
&&&{\color{red}*}&&{\color{red}\xi_2}
\end{array}
$$
The above expression for $\xi'_1$ is just the third equality in~\eqref{Rule}.
It follows that $\xi'_1$ belongs to the cluster superalgebra $A(\widetilde\Qc)$.

It was proved in~\cite{SFriZ} that the entries on the diagonals
of the superfrieze satisfy recurrence equations with coefficients
standing in the first two rows.
In particular, Lemma 2.5.3 of~\cite{SFriZ} implies the
following recurrence for the odd entries of the superfrieze:
$$
\xi'_k-\xi'_{k-1}=
-\xi_1x'_k,
\qquad
\hbox{for all}
\quad
k.
$$
One concludes, by induction on $k$,
that all of the entries $\{\xi'_1,\ldots,\xi'_{m+1}\}$ 
belong to the cluster superalgebra $A(\widetilde\Qc)$.
Again, using the induction one arrives at the same conclusion for all
parallel diagonals.

Finally, one proves in a similar way
that the entries in-between, denoted by $*$,
also belong to the cluster superalgebra $A(\widetilde\Qc)$.
\end{proof}

\section{The Laurent phenomenon and
invariance of the presymplectic form}\label{PrSect}

In this section, we prove Theorems~\ref{LeurThm} and~\ref{SymThm}.

\subsection{Proof of the Laurent phenomenon}\label{PrLPSect}

Our proof goes along the same lines as the proof
of Theorem~3.2 of~\cite{FZ1} (see also Theorem~2.1 of~\cite{FZ2}).
We will repeat this proof giving the additional arguments when necessary.

\begin{lem}
\label{AMProp}
Given an extended quiver $\widetilde{\Qc}$ such that all the $2$-paths 
in $\widetilde{\Qc}_1$ are without multiplicities,
the mutation at a given vertex $x_k$ is allowed if and only if 
for every vertex $x_\ell\in\Qc$ connected to $x_k$ by an outgoing arrow $x_k\to{}x_\ell$,
(at least) one of the following conditions is satisfied:

\begin{enumerate}

\item[(a)]
$I_k=I_\ell$;

\item[(b)]
$J_k=J_\ell$;

\item[(c)]
$I_k=J_k=\emptyset$.

\item[(d)]
$I_k=J_\ell$, and $J_k=I_\ell$;

\item[(e)]
$I_\ell=J_\ell=\emptyset$.
\end{enumerate}

\end{lem}

\begin{proof}
It is obvious that any of the conditions (a)-(e) is sufficient.
Indeed, the mutation at $x_k$ sends all $2$-paths $(\xi_i\to{}x_k\to\xi_j)$ 
to every vertex $x_\ell\in\Qc$ connected to $x_k$ by an outgoing arrow $(x_k\to{}x_\ell)$,
creating $2$-paths $(\xi_i\to{}x_\ell\to\xi_j)$.
In cases (a)-(c), at least one of the sets $I_\ell$ or $J_\ell$
remains unchanged;
in case~(d), all arrows connecting $x_\ell$ to odd vertices disappear;
in case~(e), this is also clear.

Conversely, assume that $I_k\not=I_\ell$ and $J_k\not=J_\ell$, and the sets are non-empty.
Choose $i\in{}I_k$ and $i\not\in{}I_\ell$, and $j\not\in{}J_k$ and $j\in{}J_\ell$, 
then $\xi_i$ is not connected to $\xi_j$ by a $2$-path $(\xi_i\to{}x_\ell\to\xi_j)$
this is a contradiction.
\end{proof}

\begin{ex}
\label{AllExMuT}
The following example illustrates the case (b) of the above lemma.
$$
 
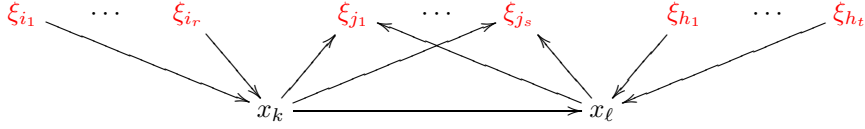
\xymatrix
 @!0 @R=1.3cm @C=1.1cm
 {
{\color{red}\xi_{i_1}}\ar@{->}[rrrd]&
\cdots&{\color{red}\xi_{i_r}}\ar@{->}[rd]&
&\ar@{<-}[ld] {\color{red}\xi_{j_1}}\ar@{<-}[rrrd]&
\cdots&\ar@{<-}[llld] {\color{red}\xi_{j_s}\ar@{<-}[rd]}&
&\ar@{->}[ld] {\color{red}\xi_{h_1}}&
\cdots&{\color{red}\xi_{h_t}\ar@{->}[llld]}
\\
&&& x_k&&&&x_\ell\ar@{<-}[llll]\\
}
$$
The mutation of this quiver at $x_k$ is allowed.
\end{ex}

Given an initial seed $(\widetilde{\Qc}^0,\,\{x^0,\xi^0\})$,
and consider a sequence of $\ell$ consecutive mutations
$$
\cdots\circ\mu_k\circ\mu_j\circ\mu_i:
(\widetilde{\Qc}^0,\,\{x^0,\xi^0\})\mapsto
(\widetilde{\Qc}^1,\,\{x^1,\xi^1\})\mapsto
\cdots\mapsto(\widetilde{\Qc}^\ell,\,\{x^\ell,\xi^\ell\}).
$$
We need to prove that the coordinates
$\{x^\ell,\xi^\ell\}$ are Laurent polynomials in~$\{x^0,\xi^0\}$.
For this end, we will use the induction on $\ell\geq3$ .

Following~\cite{FZ1,FZ2},
let us consider the following three seeds:
$$
(\widetilde{\Qc}^1,\,\{x^1,\xi^1\})=
\mu_i(\widetilde{\Qc}^0,\,\{x^0,\xi^0\}),
\qquad
(\widetilde{\Qc}^2,\,\{x^2,\xi^2\})=
\mu_j\circ\mu_i(\widetilde{\Qc}^0,\,\{x^0,\xi^0\}),
$$
and
$$
(\widetilde{\Qc}^3,\,\{x^3,\xi^3\}):=\mu_i\circ\mu_j\circ\mu_i
(\widetilde{\Qc}^0,\,\{x^0,\xi^0\}).
$$
Using the inductive assumption, we suppose that
$\{x^\ell,\xi^\ell\}$ are Laurent polynomials in~$\{x^1,\xi^1\}$ and~$\{x^2,\xi^2\}$,
since these seeds are closer to $\{x^\ell,\xi^\ell\}$ in the sequence of mutations.
Furthermore, $\{x^\ell,\xi^\ell\}$ must also be Laurent polynomials in~$\{x^3,\xi^3\}$,
since applying $\mu_i$ to $(\widetilde{\Qc}^3,\,\{x^3,\xi^3\})$, one obtains the seed
$\mu^2_i\circ\mu_j\circ\mu_i(\widetilde{\Qc}^0,\,\{x^0,\xi^0\})$
that generates the same cluster algebra as $(\widetilde{\Qc}^2,\,\{x^2,\xi^2\})$
(cf. Remark~\ref{SqMutLem}).

We now need the following analog of Lemma~3.3 of~\cite{FZ1}
(see also Lemma~2.2 of~\cite{FZ2}).

\begin{lem}
\label{CatPiL}
(i)
The coordinates~$\{x^1,\xi^1\}$, $\{x^2,\xi^2\}$ and~$\{x^3,\xi^3\}$
are Laurent polynomials in $\{x^0,\xi^0\}$.

(ii)
Furthermore, $\mathrm{gcd}(x^3_i,\,x^1_i)=\mathrm{gcd}(x^2_j,\,x^1_i)=1$,
where the greatest common divisor is taken over the ring of polynomials
in $\{x_1,\ldots,x_n\}$.
\end{lem}

\begin{proof}
Part (i).
The coordinates~$\{x^1,\xi^1\}$, $\{x^2,\xi^2\}$ are obviously Laurent polynomials in $\{x^0,\xi^0\}$.
In the case where the vertices $x_i$ and $x_j$ are not connected by an arrow,
the mutations $\mu_i$ and $\mu_j$ commute, and it is again obvious that
the coordinates $\{x^3,\xi^3\}$
are Laurent polynomials in $\{x^0,\xi^0\}$.
Let us calculate the mutation $\mu_i\circ\mu_j\circ\mu_i$ in the case where
the vertices $x_i$ and $x_j$ are connected by an arrow.
We assume that the arrow is oriented as follows $x_i\to{}x_j$
(the computations in case $x_i\leftarrow{}x_j$ are similar).

Let us start with an arbitrary extended quiver a fragment of which is as follows:
$$
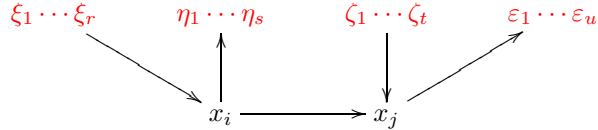
\xymatrix{
{\color{red}\xi_1\cdots\xi_r}\ar@{->}[rd]&
{\color{red}\eta_1\cdots\eta_s}\ar@{<-}[d]&
{\color{red}\zeta_1\cdots\zeta_t}\ar@{->}[d]&
{\color{red}\e_1\cdots\e_u}\ar@{<-}[ld]\\
&x_i\ar@{->}[r]& x_j&
}
$$
and show that the Laurent phenomenon occurs if and only if the mutation is allowed.
One obtains from~\eqref{Mute}:
$$
x^1_i=\frac{M^i_1+x_jM^i_2}{x_i}+
\frac{(\sum\xi)(\sum\eta)M^i_1}{x_i},
$$
where $M^i_1$ and $x_jM^i_2$ are the monomials obtained by the product of
the even coordinates connected to $x_i$ by ingoing and outgoing arrows, respectively.
Similarly,
$$
x^2_j=\frac{M^i_1M^j_1+x^1_iM^j_2}{x_j}+
\frac{(\sum\zeta)(\sum\e)M^j_1}{x_j}.
$$
The coordinates $x^1_i$ and $x^2_j$ are, of course, trivially Laurent
polynomials in the initial coordinates.
Next, after a (quite long) computation, one obtains:
$$
\begin{array}{rcl}
x^3_i&=&
\displaystyle
\frac{M^j_2+x_iM^j_1}{x_j}+
\frac{(\sum\zeta)(\sum\e)M^j_1x_i}{x_j}\\[12pt]
&&
\displaystyle
-\frac{(\sum\xi)(\sum\eta)(\sum\zeta)(\sum\e)M^i_1M^j_1x_i}{x_j(M^i_1+x_jM^i_2)}.
\end{array}
$$
(Note that the coefficient of the term $(\sum\xi)(\sum\eta)$ vanishes identically.)
We conclude that $x^3_i$ is a Laurent polynomial if and only if
the product $(\sum\xi)(\sum\eta)(\sum\zeta)(\sum\e)$ vanishes.
This is guaranteed by (and almost equivalent to) the
assumption that the mutation $\mu_1$ is allowed.

Indeed, the following statement provides with a necessary and sufficient condition
for the mutation at a vertex $x_k$ to be allowed.

The second part of the lemma stating that $x^1_i$ and $x^3_i$,
and $x^1_i$ and $x^2_j$ are coprime
follows from the similar result in the poorly even case.
Indeed, our expressions for $x^1_i,x^3_i$ and $x^2_j$ differ from their purely even 
projections by nilpotent terms.
Therefore, these expressions are coprime
if and only if their purely even parts are coprime.

Hence the Lemma.
\end{proof}

\medskip

The rest of the proof is identical to those
of Theorem~3.2 from~\cite{FZ1} (and of Theorem~2.1 from~\cite{FZ2}).
By inductive assumption,
all of the coordinates $\{x^\ell_1,\ldots,x^\ell_n,\xi^\ell_1,\ldots,\xi^\ell_m\}$
are Laurent polynomials in 
$\{x^0_1,\ldots,x^1_i,\ldots,x^0_n,\xi_1,\ldots,\xi_m\}$
and also in
$\{x^0_1,\ldots,x^3_i,\ldots,x^2_j,\ldots,x^0_n,\xi_1,\ldots,\xi_m\}$.
But the fact that the mutated coordinates are
themselves Laurent polynomials and coprime
implies that the coordinates $\{x^\ell_1,\ldots,x^\ell_n,\xi^\ell_1,\ldots,\xi^\ell_m\}$
are, indeed, Laurent polynomials in the initial coordinates.

Theorem~\ref{LeurThm} is proved.

\subsection{The presymplectic form is invariant}\label{ISFSect}
We prove Theorem \ref{SymThm} stating that,
for every extended quiver~$\widetilde{\Qc}$, the form $\om$
is invariant under mutations at even coordinates.
Our proof goes along the same lines
as the elementary proof in the classical case, 
presented in Section~\ref{GSVFSec}.

Let us choose an even vertex $x_k\in\widetilde{\Qc}$.
Perform a mutation at $x_k$, and
denote the $2$-form associated to the quiver after the mutation by $\om'$.
Our goal is to check that $\om=\om'$.

As above, in the purely even case,
collect the terms in~$\om$ containing $x_k$;
one obtains:
$$
\om_{x_k}:=
\frac{dx_k\wedge{}d\frac{x_{J_k}}{x_{I_k}}}{x\,\frac{x_{J_k}}{x_{I_k}}}+
\sum\limits_{\substack{\xi_i\to{}x_k\to\xi_j}}
\frac{d\left(
\xi_i\xi_j
\right)\wedge{}dx_k}
{x_k},
$$
where, as above, the monomial
$x_{I_k}=x_{i_1}\cdots{}x_{i_r}$ is the product of all even coordinates connected to~$x_k$
by the ingoing arrows $(x_i\to{}x_k)$, 
and where $x_{J_k}=x_{j_1}\cdots{}x_{j_s}$ is the product of all even coordinates 
connected to $x_k$
by the outgoing arrows $(x_k\to{}x_j)$.

The exchange relation~(\ref{Mute}) can be rewritten as follows:
$$
x_k=\frac{1}{\widetilde{x_k}'}
\left(
1+\frac{x_{J_k}}{x_{I_k}}+\sum\limits_{\substack{\xi_i\to{}x_k\to\xi_j}}\xi_i\xi_j
\right),
$$
where $\widetilde{x_k}':=\frac{x'_k}{x_{I_k}}$.
Substituting this expression into $\om_{x_k}$,
and applying the Leibniz rule, one obtains
$$
\frac{dx_k\wedge{}d\frac{x_{J_k}}{x_{I_k}}}{x_k\,\frac{x_{J_k}}{x_{I_k}}}=
-\frac{d\widetilde{x_k}'\wedge{}d\frac{x_{J_k}}{x_{I_k}}}{\widetilde{x_k}'\,\frac{x_{J_k}}{x_{I_k}}}+
\sum\limits_{\substack{\xi_i\to{}x_k\to\xi_j}}\frac{
d\left(\xi_i\xi_j
\right)
\wedge{}d\frac{x_{J_k}}{x_{I_k}}}{\left(1+\frac{x_{J_k}}{x_{I_k}}\right)\frac{x_{J_k}}{x_{I_k}}},
$$
for the first term, and
$$
\begin{array}{rcl}
\displaystyle
\sum\limits_{\substack{\xi_i\to{}x_k\to\xi_j}}\frac{d\left(\xi_i\xi_j
\right)\wedge{}dx_k}{x_k}
&=&
\displaystyle
-\sum\limits_{\substack{\xi_i\to{}x_k\to\xi_j}}\frac{d\left(\xi_i\xi_j
\right)\wedge{}d\widetilde{x}'}{\widetilde{x_k}'}\\[14pt]
&&\displaystyle
+
\sum\limits_{\substack{\xi_i\to{}x_k\to\xi_j}}\frac{d\left(\xi_i\xi_j
\right)\wedge{}d\frac{x_{J_k}}{x_{I_k}}}
{1+\frac{x_{J_k}}{x_{I_k}}}
\end{array}
$$
for the second term.
Collecting the above two terms, one finally gets:
$$
\displaystyle
\begin{array}{rcl}
\om_{x_k}&=&\displaystyle
-\frac{dx_k'\wedge{}d\frac{x_{J_k}}{x_{I_k}}}{x_k'\,\frac{x_{J_k}}{x_{I_k}}}-
\sum\limits_{\substack{\xi_i\to{}x_k\to\xi_j}}\frac{d\left(\xi_i\xi_j
\right)\wedge{}dx_k'}{x_k'}\\[12pt]
\displaystyle
&&
\displaystyle+
\frac{dx_{I_k}\wedge{}dx_{J_k}}{x_{I_k}\,x_{J_k}}+
\sum\limits_{\substack{\xi_i\to{}x_k\to\xi_j}}\frac{d\left(\xi_i\xi_j
\right)\wedge{}dx_{J_k}}{x_{J_k}}.
\end{array}
$$

We need to prove that
this expression coincides with the term $\om'_{x_k'}$
in the $2$-form $\om'$.
Indeed, the first two terms with minus signs correspond to the
reversing of arrows at $x_k'$;
the third term is the additional term generated by the rule (1)
of the mutation of the classical quiver $\Qc$; 
the last term is generated by the step (1*) of the mutation of the
extended quiver $\widetilde{\Qc}$.

We have proved that $\om_{x_k}=\om'_{x_k'}$,
and hence $\om=\om'$, since all the other terms of the $2$-form remain unchanged under
the mutation at $x_k$.

Theorem \ref{SymThm} is proved.

\section{An application: extended Somos-$4$ sequence}\label{ApplSect}
The goal of this section is to show that
cluster superalgebras can have very concrete applications.
We present two sequences of numbers constructed as extensions
of the well-known {\it Somos-$4$ sequence}.
We deduce from the Laurent phenomenon for superalgebras
that all entries of these sequences are integers.

Recall that the Somos-$4$ sequence is the sequence of numbers
$(a_n)_{n\in\N}$ defined by the recurrence
$$
a_{n+4}=\frac{a_{n+1}a_{n+3}+a_{n+2}^2}{a_n},
$$
and the initial conditions $a_1=a_2=a_3=a_4=1$,
and it turns out that the sequence is integer.
The first values in the sequence are: 
$1, 1, 1, 1, 2, 3, 7, 23, 59, 314, 1529, 8209,83313,620297,7869898\ldots$

A conceptual proof of integrality,
obtained in~\cite{FZ2}, is an application of the Laurent phenomenon.
Indeed, consider one of the following quiver (that differ only by orientation):
$$
 
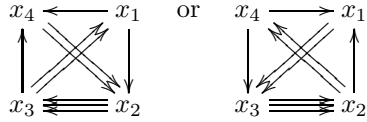
\xymatrix{
x_4\ar@<3pt>@{->}[rd]\ar@{->}[rd]\ar@{<-}[d]&
x_1\ar@<-3pt>@{<-}[ld]\ar@{<-}[ld]\ar@{->}[d]\ar@{->}[l]\\
x_3&x_2\ar@{->}[l]\ar@<-2pt>@{->}[l]\ar@<2pt>@{->}[l]
}
\quad\hbox{or}\quad
 \xymatrix{
x_4\ar@<3pt>@{<-}[rd]\ar@{<-}[rd]\ar@{->}[d]&
x_1\ar@<-3pt>@{->}[ld]\ar@{->}[ld]\ar@{<-}[d]\ar@{<-}[l]\\
x_3&x_2\ar@{<-}[l]\ar@<-2pt>@{<-}[l]\ar@<2pt>@{<-}[l]
}
$$
Each of them performs a cyclic rotation under the mutation at $x_1$:
$$
 
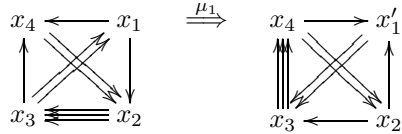
\xymatrix{
x_4\ar@<3pt>@{->}[rd]\ar@{->}[rd]\ar@{<-}[d]&
x_1\ar@<-3pt>@{<-}[ld]\ar@{<-}[ld]\ar@{->}[d]\ar@{->}[l]\\
x_3&x_2\ar@{->}[l]\ar@<-2pt>@{->}[l]\ar@<2pt>@{->}[l]
}
\quad
\stackrel{\mu_1}{\Longrightarrow}
\quad
 \xymatrix{
x_4\ar@<3pt>@{->}[rd]\ar@{->}[rd]\ar@{<-}[d]\ar@<-2pt>@{<-}[d]\ar@<2pt>@{<-}[d]&
x_1'\ar@<-3pt>@{->}[ld]\ar@{->}[ld]\ar@{<-}[d]\ar@{<-}[l]\\
x_3&x_2\ar@{->}[l]
}
$$
that follows immediately from the definition of quiver mutations (see Section~\ref{ClDeS}).
The exchange relation reads:
$x_1x_1'=x_2x_4+x_3^2$.
Choosing $x_5:=x_1'$, and performing an infinite cyclic process of mutations:
$x_6=x_2',x_7=x_3',\ldots$, one obtains a sequence of rational functions $x_n$ in
the initial variables $x_1,x_2,x_3,x_4$, satisfying the
Somos recurrence.
By Laurent phenomenon, every function~$x_n$ is a Laurent polynomial.
Therefore, choosing the initial values $x_1=x_2=x_3=x_4=1$,
one obtains an integer sequence.

More interesting examples of cyclically invariant quivers
and corresponding sequences can be found in~\cite{FH},
and of sequences with the Laurent property in~\cite{FZ2}
and~\cite{HW}.

Our extension of the Somos-$4$ sequence is as follows.
Consider a pair of sequences of numbers, $(a_n), (b_n)_{n\in\N}$,
that we write in the form of a linear function in a formal variable $\e$:
$$
A_n=a_n+b_n\e,
$$
such that $\e^2=0$.
This expression is known under the name of ``dual numbers''
(invented by Clifford and used by Study, Grassmann, Blaschke,... 
for geometric purposes).
The following statement is a corollary of Theorem~\ref{LeurThm}.

\begin{cor}
Suppose that the sequence $(A_n)$ satisfies the recurrence:
\begin{equation}
\label{SSomos}
A_{n+4}=\frac{A_{n+1}A_{n+3}+A_{n+2}^2\left(1+\e\right)}{A_n},
\end{equation}
with the initial conditions:
$$
A_1=A_2=A_3=A_4=1,
$$
i.e., with $a_1=a_2=a_3=a_4=1$ and $b_1=b_2=b_3=b_4=0$.
Then both sequences, $(a_n)$ and $(b_n)$, are integer.
\end{cor}

\begin{proof}
The following extended quiver:
$$
 
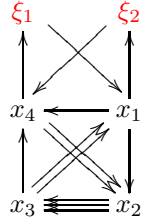
\xymatrix{
{\color{red}\xi_1}\ar@<3pt>@{->}[rd]\ar@{<-}[d]&
{\color{red}\xi_2}\ar@<-2pt>@{->}[ld]\ar@{<-}[d]\\
x_4\ar@<3pt>@{->}[rd]\ar@{->}[rd]\ar@{<-}[d]&
x_1\ar@<-3pt>@{<-}[ld]\ar@{<-}[ld]\ar@{->}[d]\ar@{->}[l]\\
x_3&x_2\ar@{->}[l]\ar@<-2pt>@{->}[l]\ar@<2pt>@{->}[l]
}
$$
performs a cyclic rotation under the mutation at $x_1$,
cf. Definition~\ref{MutDef}.
The exchange relation~(\ref{Mute}) reads
$x_1x_1'=x_2x_4+\left(1+\xi_1\xi_2\right)x_3^2$. 
Setting $\e=\xi_1\xi_2$, and applying an infinite cyclic sequence of mutations,
one obtains a sequence of rational functions $x_n(x_1,x_2,x_3,x_4,\xi_1,\xi_2)$
satisfying the recursion~(\ref{SSomos}).
The result then follows from the Laurent phenomenon.
\end{proof}

The sequence $(a_n)$ is obviously just the Somos-$4$ sequence.
The first values of the sequence $(b_n)$ are:
$0,0,0,0,1,2,10,48,160,1273,7346,51394,645078,5477318,87284761\ldots$
This sequence cannot be found in OEIS.

\begin{rem}
Rewriting~(\ref{SSomos}) componentwise, one gets the following, more
explicit, but less pleasant recurrence for the sequence $(b_n)$:
$$
b_{n+4}=\frac{a_{n+1}b_{n+3}+
2a_{n+2}b_{n+2}+a_{n+3}b_{n+1}+a_{n+2}^2}{a_n}-
\frac{\left(a_{n+3}a_{n+1}+a_{n+2}^2\right)b_n}{a^2_n}.
$$
Yet another relation which is also equivalent to the above recurrence is:
$$
b_{n+4}a_n+b_na_{n+4}=a_{n+1}b_{n+3}+
2a_{n+2}b_{n+2}+a_{n+3}b_{n+1}+a_{n+2}^2
$$
(communicated to me by Michael Somos).
\end{rem}

Consider, furthermore, the following extended quiver:
$$
 
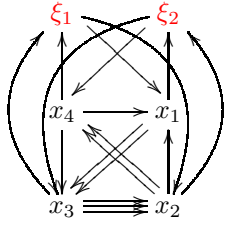
\xymatrix{
{\color{red}\xi_1}\ar@<3pt>@{->}[rd]\ar@{<-}[d]
\ar@/^2.2pc/[rdd]&
{\color{red}\xi_2}\ar@<-2pt>@{->}[ld]\ar@{<-}[d]\ar@/^-2.2pc/[ldd]\\
x_4\ar@<3pt>@{<-}[rd]\ar@{<-}[rd]\ar@{->}[d]&
x_1\ar@<-3pt>@{->}[ld]\ar@{->}[ld]\ar@{<-}[d]\ar@{<-}[l]\\
x_3\ar@/^1.7pc/[uu]&x_2\ar@/^-1.7pc/[uu]\ar@{<-}[l]\ar@<-2pt>@{<-}[l]\ar@<2pt>@{<-}[l]
}
$$
It is not difficult to see that under the cyclic sequence of mutations,
$\mu_1,\mu_2,\mu_3,\mu_4,\mu_1,\ldots$ this quiver
also rotates cyclically, and the exchange relation~(\ref{Mute}) reads this time:
$x_1x_1'=x_3^2+\left(1+\xi_1\xi_2\right)x_2x_4$.
One then obtains another extension, $\tilde A_n=a_n+\tilde b_n\e$, of the Somos-$4$ sequence,
satisfying the recursion
$$
\tilde A_{n+4}=\frac{\tilde A_{n+2}^2+\tilde A_{n+1}\tilde A_{n+3}\left(1+\e\right)}{\tilde A_n}.
$$
It follows from Theorem~\ref{LeurThm}, that the sequence $(\tilde b_n)_{n\in\N}$ is, again, integer.
The first values of this sequence are:
$0,0,0,0,1,3,10,59,198,1387,9389,57983,752301,6851887,97297759\ldots$

\section{Variations on Fibonacci: extended Kronecker quiver}\label{FiboSect}

In this section, we give another,
more elementary illustration of the Laurent phenomenon
of cluster superalgebras.

The classical Fibonacci sequence
$(F_n)=1,\,1,\,2,\,3,\,5,\,8,\,13,\,21,\,34,\,55,\,89,\,144,\,233,\,377,\ldots$
can also be understood within the cluster algebra framework.
Consider the sequence $(a_n)$,
satisfying the recurrence
$$
a_{n+2}=\frac{a_{n+1}^2+1}{a_n},
$$
and the initial conditions $a_0=a_1=1$.
The sequence $(a_n)_{n\in\N}$
is then the bisected Fibonacci sequence:
$a_n=F_{2n-1}$.
The above recurrence
is known as Cassini's identity.

Our methods lead to a family of extensions of the sequence $(a_n)$.
As before, we consider another sequence $(b_n)$,
coupled with $(a_n)$ in the form of a sequence of dual numbers: 
$A_n=a_n+b_n\e$, where~$\e^2=0$.
One of these extensions is defined by the recurrence
\begin{equation}
\label{SupCass}
A_{n+2}A_n=A_{n+1}^2\left(1+(-1)^{\frac{(n+1)(n+2)}{2}}\e\right)+1,
\end{equation}
and the initial conditions $A_0=A_1=1$.
Equivalently,
$$
b_{n+2}a_n+b_{n}a_{n+2}=
2b_{n+1}a_{n+1}+(-1)^{\frac{(n+1)(n+2)}{2}}a_{n+1}^2,
$$
with $b_0=b_1=0$.

It turns out that the sequence $(b_n)_{n\in\N}$ then also consists of Fibonacci numbers,
but this time with even indices, and taken in a surprising order:
$$
\begin{array}{r|c|c|c|c|c|c|c|cc}
n&0&1&2&3&4&5&6&7&\cdots\\[2pt]
\hline
a_n&1&1&2&5&13&34&89&233&\cdots\\[2pt]
\hline
b_n&0&0&1&8&21&21&55&377&\cdots
\end{array}
$$
More precisely, 
$$
b_n=\left\{
\begin{array}{ll}
F_{2n},&n\equiv0,3\mod4,\\
F_{2n-2},&n\equiv1,2\mod4.
\end{array}
\right.
$$
Our next goal is to explain the origin of this sequence and its relation
to cluster superalgebra.

The classical sequence $(a_n)$ can be generated by consecutive mutations
$\mu_0,\mu_1,\mu_0,\mu_1,\ldots$ of the quiver
with two vertices and two arrows:
$$
 \xymatrix{
x_0&x_1\ar@{->}[l]\ar@<2pt>@{->}[l]
},
$$
called the $2$-{\it Kronecker quiver}\footnote{
Note that the Fibonacci sequence is also related to the 
$3$-Kronecker quiver, see~\cite{FR}.}.
Indeed, the quiver mutation is just the orientation reversing:
$$
 \xymatrix{
x_0&x_1\ar@{->}[l]\ar@<2pt>@{->}[l]
}
\quad
\stackrel{\mu_0}{\Longrightarrow}
\quad
 \xymatrix{
x_0'&x_1\ar@{<-}[l]\ar@<2pt>@{<-}[l]
},
$$
while the exchange relations read:
$
x_0x_0'=x_1^2+1.
$

Let us consider the following extended Kronecker quiver:
$$
 
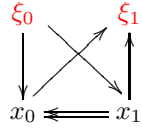
\xymatrix{
{\color{red}\xi_0}\ar@<3pt>@{->}[rd]\ar@{->}[d]&
{\color{red}\xi_1}\ar@<-2pt>@{<-}[ld]\ar@{<-}[d]\\
x_0&x_1\ar@{->}[l]\ar@<2pt>@{->}[l]
}
$$
The sequence of consecutive mutations of this
extended quiver is $4$-periodic:
$$
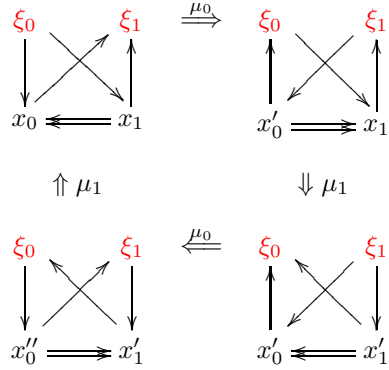
\begin{array}{ccc}
 \xymatrix{
{\color{red}\xi_0}\ar@<3pt>@{->}[rd]\ar@{->}[d]&
{\color{red}\xi_1}\ar@<-2pt>@{<-}[ld]\ar@{<-}[d]\\
x_0&x_1\ar@{->}[l]\ar@<2pt>@{->}[l]
}
&
\stackrel{\mu_0}{\Longrightarrow}
&
 \xymatrix{
{\color{red}\xi_0}\ar@<3pt>@{->}[rd]\ar@{<-}[d]&
{\color{red}\xi_1}\ar@<-2pt>@{->}[ld]\ar@{<-}[d]\\
x_0'&x_1\ar@{<-}[l]\ar@<2pt>@{<-}[l]
}\\[50pt]
\Uparrow\mu_1&&\Downarrow\mu_1\\[10pt]
 \xymatrix{
{\color{red}\xi_0}\ar@<3pt>@{<-}[rd]\ar@{->}[d]&
{\color{red}\xi_1}\ar@<-2pt>@{<-}[ld]\ar@{->}[d]\\
x_0''&x_1'\ar@{<-}[l]\ar@<2pt>@{<-}[l]
}
&
\stackrel{\mu_0}{\Longleftarrow}
&
 \xymatrix{
{\color{red}\xi_0}\ar@<3pt>@{<-}[rd]\ar@{<-}[d]&
{\color{red}\xi_1}\ar@<-2pt>@{->}[ld]\ar@{->}[d]\\
x_0'&x_1'\ar@{->}[l]\ar@<2pt>@{->}[l]
}
\end{array}
$$
The exchange relations are:
$$
x^{(n+1)}_0x^{(n)}_0
={x^{(n)}_{1}}^2\left(1+(-1)^{n}\xi_0\xi_1\right)+1,
\qquad
x^{(n+1)}_1x^{(n)}_1=
{x^{(n+1)}_{0}}^2\left(1+(-1)^{n}\xi_0\xi_1\right)+1,
$$
which corresponds to the recurrence~(\ref{SupCass}) with $\e=\xi_0\xi_1$.

More generally, starting with the extended Kronecker quiver
with two odd variables related by an arbitrary number $k$ $2$-path through $x_0$,
and $\ell$ $2$-path through $x_1$, one obtains a two-parameter family
of integer sequences extending $(a_n)$:
$$
\begin{array}{r|c|c|c|c|c|c|c|cc}
n&0&1&2&3&4&5&6&7&\cdots\\[2pt]
\hline
a_n&1&1&2&5&13&34&89&233&\cdots\\[2pt]
\hline
b_n&0&0&k&4(\ell+k)&20\ell+k&43\ell-22k&88\ell-33k&399\ell-22k&\cdots
\end{array}
$$
The previously considered case of $(b_n)$ given by Fibonacci numbers,
corresponds to $k=\ell=1$.

\bigskip
{\bf Acknowledgements}.
I am grateful to Sophie Morier-Genoud and Sergei Tabachnikov 
for a number of fruitful discussions,
our joint work on superfriezes~\cite{SFriZ} is the first step of the present work.
I am grateful to Michael Somos for his interest, enlightening discussions and the computer programs.
I am also pleased to thank 
Fr\'ed\'eric Chapoton, Sergei Fomin,
Michael Gekhtman, Dimitry Leites, Yann Palu and
Xiuling Wang
 for interesting discussions. 
Special thanks to Sophie Morier-Genoud and Dimitry Leites
 for a careful reading of the first version of this paper.
This work was partially supported by the PICS05974 ``PENTAFRIZ'' of CNRS.

\end{document}